\documentclass[12pt,thmsa]{article}
\usepackage{amssymb}
\usepackage{amsfonts}
\usepackage{enumerate}
\usepackage{amsmath}
\usepackage{cite}
\usepackage[top=3.8cm,bottom=3.8cm,textwidth=16cm,centering,a4paper]{geometry}
\usepackage{tikz}
\usepackage{caption2}
\usepackage{subfigure}
\usepackage{tcolorbox}

\newtheorem{theorem}{Theorem}[section]

\newtheorem{remark}[theorem]{Remark}

\newtheorem{lemma}[theorem]{Lemma}

\newtheorem{definition}[theorem]{Definition}
\newenvironment{proof}[1][Proof]{\noindent\textbf{#1.} }{\ \rule{0.5em}{0.5em}}

\begin{document}

\title{Normalized solutions for NLS equations with potential on
bounded domains: Ground states and multiplicity}
\date{}
\author{He Zhang$^{a}$\thanks{
E-mail address: hehzhang@163.com(H. Zhang)}, Haibo Chen$^{a}$\thanks{%
E-mail address: math\_chb@163.com(H. Chen)}, Shuai Yao$^{b}$\thanks{%
E-mail address: shyao@sdut.edu.cn (S. Yao)}, Juntao Sun$^{b}$\thanks{%
E-mail address: jtsun@sdut.edu.cn(J. Sun)}\\
{\footnotesize $^a$\emph{School of Mathematics and Statistics, Central South
University, Changsha 410083, PR China }}\\
{\footnotesize $^{b}$\emph{School of Mathematics and Statistics, Shandong
University of Technology, Zibo 255049, PR China }}}
\maketitle

\begin{abstract}
We investigate normalized solutions for a class of nonlinear Schr\"{o}dinger (NLS) equations with potential
$V$ and inhomogeneous nonlinearity $g(|u|)u=|u|^{q-2}u+\beta |u|^{p-2}u$ on a bounded domain $\Omega$. Firstly, when $2+\frac{4}{N}<q<p\leq2^*:=\frac{2N}{N-2}$ and $\beta=-1$, under an explicit smallness assumption on $V$, we prove the existence of a global minimum solution and a high-energy solution if the mass is large enough. For this case we do not require that $\Omega$ is star-shaped, which partly solves an open problem by Bartsch et al.
[Math. Ann. 390 (2024) 4813--4859]. Moreover, we find that the global minimizer also exists although the nonlinearity is $L^2$-supercritical. Secondly, when $2<q<2+\frac{4}{N}<p=2^*$ and $\beta=1$, under the smallness and some extra assumptions on $V$, we prove the existence of a ground state and a high-energy solution if $\Omega$ is star-shaped and the mass is small enough. It seems to be new in the study of normalized ground state in the context of the Br\'{e}zis-Nirenberg problem, even for the autonomous case of $V(x)\equiv0$.

\end{abstract}

\textbf{Keywords:} NLS equations; Ground states; Multiplicity; The monotonicity trick

\textbf{MSC(2020):} 35J20, 35J60,35Q55

\section{Introduction}

Our starting point is to consider standing waves of the NLS equation
\begin{equation}
\left\{
\begin{array}{ll}
i\partial _{t}\Phi +\Delta \Phi -V(x)\Phi +g(|\Phi |)\Phi =0, & \ \forall
(t,x)\in \mathbb{R}\times \Omega , \\
\Phi (t,x)=0, & \text{ }\forall (t,x)\in \mathbb{R}\times \partial
\Omega,%
\end{array}%
\right.  \label{e1}
\end{equation}%
where $\Omega \subset\mathbb{R}^{N}(N\geq 1)$ is a bounded domain, $\Phi
=\Phi (t,x)\in\mathbb{C}$ is a wave function and $V(x)$ is an external potential. It is well known
that the NLS equation is a fundamental equation in quantum mechanics, which
appears in different physical contexts, such as the nonlinear optics \cite%
{A1,FM} and the theory of Bose-Einstein condensation \cite{AEMWC,ESY,G}.

As we all know, two quantities are conserved along trajectories of (\ref{e1}):
the energy%
\begin{equation*}
\Sigma (\Phi )=\frac{1}{2}\int_{\Omega}|\nabla \Phi |^{2}dx+\frac{1%
}{2}\int_{\Omega }V(x)|\Phi |^{2}dx-\int_{\Omega}G(|\Phi
|^{2})dx
\end{equation*}%
with $G(s)=\int_{0}^{s}g(|t|)tdt,$ and the mass%
\begin{equation*}
M(\Phi )=\int_{\Omega}|\Phi |^{2}dx.
\end{equation*}%
A standing wave is a solution of the form $\Phi (t,x)=e^{-i\lambda t}u(x)$,
where $\lambda \in \mathbb{R}$. Then the real valued function $u$ satisfies
the elliptic equation%
\begin{equation}
\left\{
\begin{array}{ll}
-\Delta u+V(x)u+\lambda u=g(|u|)u, & \ \text{in }\Omega, \\
u(x)=0, & \text{ on }\partial \Omega.%
\end{array}%
\right.  \label{e2}
\end{equation}%
To study solutions of (\ref{e2}), one can either consider the frequency $%
\lambda \in
\mathbb{R}
$ to be given, or to be an unknown of the problem. In the latter case, $%
\lambda $ appears as a Lagrange multiplier and $L^{2}$-norms of solutions
are prescribed, i.e.
\begin{equation}
\int_{\Omega }|u|^{2}dx=a>0,  \label{e3}
\end{equation}%
which are usually called normalized solutions. This study seems to be
particularly meaningful from the physical point of view, since standing
waves of (\ref{e1}) conserve their mass along time.

After the pioneering contribution by Jeanjean \cite{J1}, the study of
normalized solutions for NLS equation or system on the whole space $\mathbb{R%
}^{N}$ has attracted much attention in recent years regarding different
types of nonlinearities. We refer the readers to \cite%
{BM,BS,BJL,BCJS,CJS,DST2,JL,JZZ,MRV,S3,S4,SYZ,WW,YCRS,YCS}, in which some methods related to the Pohozaev identity were introduced, such as the mountain pass theorem
and the Pohozaev manifold, since $\mathbb{R}^{N}$ is invariant under
translation and dilation.

Compared to the study of normalized solutions in the case of the whole
space, there seem only a few results focusing on normalized solutions in the
case of bounded domains. In fact, these two settings are rather different in
nature: each one requires a specific approach, and the results are in
general not comparable. For example, the Pohozaev manifold is not
available when working on bounded domains $\Omega $, since the invariance
under translation and dilation is lost. Up to our knowledge, the first
attempt to deal with this case was made by Noris, Tavares and Verzini \cite%
{NTV1}, where a local minimizer and a mountain pass solution of (\ref{e2}%
)--(\ref{e3}) were found when $\Omega $ is a ball, $V(x)\equiv 0$ and
the pure power nonlinearity $g(|u|)u=|u|^{p-2}u$ is Sobolev subcritical,
i.e. $2<p<2^{\ast }:=\frac{2N}{N-2}$. Subsequently, the existence of a local
minimizer was extended to the Sobolev subcritical case on generic bounded
domains \cite{PV}, and to the Sobolev critical case ($p=2^{\ast }$) on
bounded domains \cite{NTV2}. Moreover, the existence of a mountain pass
solution was extended to the Sobolev critical case on bounded domains \cite%
{PVY} and on star-shaped bounded domains \cite{SZ} almost simultaneously.
Also, in \cite{SZ} the authors proved that the local minimizer obtained in
\cite{PV} is a ground state with positive energy on star-shaped bounded
domains. Note that in all these papers the equation is required to be
autonomous.

In recent two papers \cite{BQZ,QZ}, normalized solutions to NLS equations
with potential and inhomogeneous nonlinearity were studied. In particular,
Bartsch, Qi and Zou \cite{BQZ} obtained the existence and multiplicity
results for (\ref{e2})--(\ref{e3}) on large bounded star-shaped domains
when $V(x)$ is a bounded potential and $g(|u|)u=|u|^{q-2}u+\beta |u|^{p-2}u$
with $2<p<q^{\ast }:=2+\frac{4}{N}<q<2^{\ast }$ and $\beta \in\mathbb{R}$.
Due to the presence of the potential and the bounded domain,
the standard approach related to the Pohozaev identity cannot be used.
In view of this, they applied the monotonicity trick developed by \cite{BCJS,CJS}.

Inspired by \cite{BQZ,QZ}, in this paper we are likewise interested in
looking for normalized solutions to a class of NLS\ equations with potential
and inhomogeneous combined nonlinearity on bounded domains. However,
unlike the study conducted in \cite{BQZ,QZ}, this work aims to find a ground
state and a high-energy solution of (\ref{e2})--(\ref{e3}) under two
types of combined supercritical nonlinearities different from those in \cite%
{BQZ,QZ}. Moreover, we are also dedicated to relaxing the restriction on the
star-shaped domains, which can partly solve an open problem raised in \cite%
{BQZ}. Specifically, for $a>0,$ the problem considered in this study is as
follows:%
\begin{equation}
\left\{
\begin{array}{ll}
-\Delta u+V(x)u+\lambda u=|u|^{q-2}u+\beta |u|^{p-2}u & \ \text{in}\ \Omega
_{r}, \\
u\in H_{0}^{1}(\Omega _{r}),\ \int_{\Omega _{r}}|u|^{2}dx=a. &
\end{array}%
\right.  \label{e4}
\end{equation}%
where $\Omega _{r}:=\{rx\in \mathbb{R}^{N}:x\in \Omega \}$ with $r>0$ and $%
\Omega $ being a bounded domain, and $V(x)$ is a potential on $\mathbb{R}%
^{N} $. The exponents $p,q$ and the parameter $\beta $ satisfy either $%
q^{\ast }<q<p\leq 2^{\ast }$ and $\beta =-1,$ or $2<q<q^{\ast }<p=2^{\ast }$
and $\beta =1.$

Solutions of problem (\ref{e4}) can be identified with critical points of
the energy functional $I_{r}:H_{0}^{1}(\Omega _{r})\rightarrow \mathbb{R}$
defined by
\begin{equation*}
I_{r}(u)=\frac{1}{2}\int_{\Omega _{r}}|\nabla u|^{2}dx+\frac{1}{2}%
\int_{\Omega _{r}}V(x)u^{2}dx-\frac{1}{q}\int_{\Omega _{r}}|u|^{q}dx-\frac{%
\beta }{p}\int_{\Omega _{r}}|u|^{p}dx
\end{equation*}%
on the constraint%
\begin{equation*}
S_{r,a}:=\left\{ u\in H_{0}^{1}(\Omega _{r}):\int_{\Omega
_{r}}|u|^{2}dx=a\right\} .
\end{equation*}%
To find normalized solutions of problem (\ref{e4}), we have to face many difficulties, and new ideas and
techniques need to be explored. More details will be discussed in the next
subsection.

\subsection{Main results}

Before stating our main results, we introduce some notations and assumptions.

\begin{definition}
\label{D1.1}We say that a solution $u\in S_{r,a}$ of problem (\ref{e4}) is a
ground state if it possesses the minimal energy among all solutions in $%
S_{r,a}$, i.e. if
\begin{equation*}
I_{r}(u)=\inf \left\{ I_{r}(v):v\in S_{r,a},\text{ }(I_{r}|_{S_{r,a}})^{%
\prime }(v)=0\right\} .
\end{equation*}
\end{definition}

Let $s_{+}=\max \{s,0\}$ and $s_{-}:=\min \{s,0\}$ for $s\in \mathbb{R}.$
The Aubin-Talenti constant \cite{A2} is denoted by $S$, that is, $S$ is the
best constant in the Sobolev embedding $D^{1,2}(\mathbb{R}%
^{N})\hookrightarrow L^{2^{\ast }}(\mathbb{R}^{N})$. Denote $C_{s}$ be the
best constant in the Gagliardo-Nirenberg inequality
\begin{equation*}
\Vert u\Vert _{s}^{s}\leq C_{s}\Vert u\Vert _{2}^{\frac{2s-N(s-2)}{2}}\Vert
\nabla u\Vert _{2}^{\frac{N(s-2)}{2}}\text{ for }2<s<2^{\ast }.
\end{equation*}%
Note that $C_{2^{\ast }}=S^{-2^{\ast }/2}$. Let $\theta $ be the principal
eigenvalue of $-\Delta $ with Dirichlet boundary conditions in $\Omega $ and
$|\Omega |$ be the volume of $\Omega .$ Set%
\begin{equation*}
a_{V}:=\left[ \frac{\theta (N(p-2)-4)(1+\Vert V\Vert _{N/2}S^{-1})}{N(p-q)}%
\right] ^{\frac{N}{2}}\left[ \frac{2qC_{p}\theta ^{\frac{N(p-2)}{4}}|\Omega
|^{\frac{q-2}{2}}(N(p-2)-4)}{p(N(q-2)-4)}\right] ^{\frac{N(q-2)-4}{2(p-q)}}
\end{equation*}%
and%
\begin{equation*}
\widetilde{a}_{V}:=\left[ \frac{4(1-\Vert V_{-}\Vert _{N/2}S^{-1})}{%
N(2N-q(N-2))}\right] ^{\frac{4N}{(N-2)(N(q-2)-4)}}\left[ \frac{16q}{%
C_{q}N(q-2)(4-N(q-2))(N-2)S^{2^{\ast }/2}}\right] ^{\frac{4N}{%
(N-2)(2N-q(N-2))}}
\end{equation*}

Assume that the potential $V(x)$ satisfies the following assumption:

\begin{itemize}
\item[$(V_{0})$] $V\in C(\mathbb{R}^{N})\cap L^{N/2}(\mathbb{R}^{N})$ is
bounded and $\Vert V_{-}\Vert _{N/2}<S.$
\end{itemize}

We now summarize our main results.

\begin{theorem}
\label{T1.1} Let $N\geq 3,q^{\ast }<q<p\leq 2^{\ast },\beta =-1$ and $\Omega
$ be a bounded smooth domain. Assume that condition $(V_{0})$ holds. Then
for any $a>a_{V},$ the following statements are true.

\begin{itemize}
\item[$(i)$] There exists $r_{a}>0$ such that for any $r>r_{a},$ problem (\ref%
{e4}) has a mountain pass type solution $u_{r}\in H_{0}^{1}(\Omega _{r})$
satisfying $u_{r}>0$ in $\Omega _{r}$ and $I_{r}(u_{r})>0$ for some Lagrange
multiplier $\lambda _{r}\in \mathbb{R}$.

\item[$(ii)$] There exists $\overline{r}_{a}>0$ such that for any $r>\overline{r}%
_{a},$ problem (\ref{e4}) has a global minimum type solution $\overline{%
u}_{r}\in H_{0}^{1}(\Omega _{r})$ satisfying $\overline{u}_{r}>0$ in $\Omega
_{r}$ and $I_{r}(\overline{u}_{r})<0$ for some Lagrange multiplier $%
\overline{\lambda }_{r}\in \mathbb{R}$. In particular,
\begin{equation*}
\Vert \nabla \overline{u}_{r}\Vert _{2}\geq \left[ \frac{2q}{N(q-2)C_{q}}%
\left( 1-\Vert V_{-}\Vert _{N/2}S^{-1}\right) a^{\frac{q(N-2)-2N}{4}}\right]
^{\frac{2}{N(q-2)-4}}.
\end{equation*}%
Moreover, if in addition $V\in C^{1}(\mathbb{R}^{N})$ satisfies $q\Vert
\widetilde{V}_{+}\Vert _{N/2}+N(q-2)\Vert V_{+}\Vert _{N/2}<S(2N-q(N-2))$,
here $\widetilde{V}(x):=\nabla V(x)\cdot x,$ then we have $\lim
\inf\limits_{r\rightarrow \infty }\overline{\lambda }_{r}>0$.

\item[$(iii)$] There exists $\vartheta_{a}>0$ such that
\begin{equation*}
\limsup_{r\rightarrow \infty }\max_{x\in \Omega _{r}}u_{r}<\vartheta_{a}\ \text{and}%
\ \limsup_{r\rightarrow \infty }\max_{x\in \Omega _{r}}\overline{u}%
_{r}<\vartheta_{a}.
\end{equation*}
\end{itemize}
\end{theorem}

\begin{remark}
\label{R1.1} Clearly, $\overline{u}_{r}$ is a ground state and $u_{r}$ is a
high-energy solution of problem (\ref{e4}).
\end{remark}

To prove the solution $u_{r}$ with positive energy in Theorem \ref{T1.1}, we
have to face two difficulties. The first one is how to construct a bounded
(PS)-sequence at the mountain pass level, since $\Omega $ is not scaling
invariant. To overcome this issue, an effective strategy is to use the
montonicity trick without blow-up analysis developed by \cite{BCJS,CJS}. The
second one is how to obtain the uniform boundedness of the solutions of the
modified equation on the parameter. We note that it can be solved by
assuming that $\Omega $ is star-shaped in \cite{BQZ}. However, in our proof
we can prove this claim only by describing the geometric properties of
the energy functional related to the modified equation, not assuming that $%
\Omega $ is star-shaped, which partly solves an open problem raised in \cite%
{BQZ}.

The second solution $\overline{u}_{r}$ with negative level in Theorem \ref%
{T1.1} is obtained as a global minimizer of the energy functional $I_{r}$
via the minimization method. This means that the global minimizer also
exists although the nonlinearity is $L^{2}$-supercritical, which is an
unexpected result.

Next, we turn to consider another case.

\begin{theorem}
\label{T1.4} Let $N\geq 3,2<q<q^{\ast }<p=2^{\ast },\beta =1$ and $\Omega $
be a bounded smooth domain. Assume that condition $(V_{0})$ holds. Then for
any $0<a<\widetilde{a}_{V},$ the following statements are true.

\begin{itemize}
\item[$(i)$] There exists $\widetilde{r}_{a}>0$ such that for any $r>\widetilde{r%
}_{a},$ problem (\ref{e4}) has a local minimum type solution $\widetilde{u}%
_{r}\in H_{0}^{1}(\Omega _{r})$ satisfying $\widetilde{u}_{r}>0$ in $\Omega
_{r}$ and $I_{r}(\widetilde{u}_{r})<0$ for some Lagrange multiplier $%
\widetilde{\lambda }_{r}\in \mathbb{R}$. Moreover, there holds $\liminf\limits_{r\rightarrow \infty }\widetilde{%
\lambda }_{r}>0$.

\item[$(ii)$] If in addition $\Omega $ is a star-shaped domain with respect
to $0$ and $V\in C^{1}(\mathbb{R}^{N}),$ then there exists $0<\widetilde{a}%
_{\ast }\leq \widetilde{a}_{V}$ such that for any $0<a<\widetilde{a}_{\ast }$
and $r>\widetilde{r}_{a},$ the solution $\widetilde{u}_{r}$ obtained by $(i)$ is a ground state.
\end{itemize}
\end{theorem}

In Theorem \ref{T1.4}, the solution $\widetilde{u}_{r}$ with negative energy
is obtained as a local minimizer of the energy functional $I_{r}$ by using
the usual arguments in \cite{BQZ,S3,S4}. If we further restrict the mass $a$
to be small enough, then $\widetilde{u}_{r}$ is a ground state based on the
proof by contradiction. To the best of our knowledge, it is new in the study of normalized ground state
in the context of the Br\'{e}zis-Nirenberg problem, even for the autonomous case of $V(x)\equiv0$.

\begin{theorem}
\label{T1.5} Let $N\geq 3,2<q<q^{\ast }<p=2^{\ast },\beta =1$ and $\Omega $
be a bounded smooth star-shaped domain with respect to $0$. Assume that
condition $(V_{0})$ holds and $V\in C^{1}(\mathbb{R}^{N})$, $\widetilde{V}(x)$
is bounded on $\mathbb{R}^{N}$. Then for any $0<a<\widetilde{a}_{\ast },$
there exists $\widehat{r}_{a}>0$ such that for any $r>\widehat{r}_{a},$ problem (%
\ref{e4}) has a high-erengy solution $\widehat{u}_{r}\in H_{0}^{1}(\Omega
_{r})$ satisfying $\widehat{u}_{r}>0$ in $\Omega _{r}$ and $I_{r}(\widehat{u}%
_{r})>0$ for some Lagrange multiplier $\widehat{\lambda }_{r}\in \mathbb{R}
$. Moreover, there holds $\Vert \nabla \widehat{u}_{r}\Vert _{2}\rightarrow
\infty $ as $a\rightarrow 0.$
\end{theorem}
\begin{remark}
\label{R1.2} $(i)$ By Theorems \ref{T1.4} and \ref{T1.5}, we extend the results in
\cite[Theorems 1.3 and 1.4]{BQZ} to the Sobolev critical case.

$(ii)$ By comparing Theorem \ref{T1.1} with Theorems \ref{T1.4} and \ref{T1.5}, we find that the mass $a$ is required to be different size, although they have same numbers of normalized solutions.
\end{remark}

To find the solution $\widehat{u}_{r}$ in Theorem \ref{T1.5}, we likewise
use the montonicity trick to construct a bounded (PS)-sequence at the
mountain pass level. However, due to the presence of the positive Sobolev
critical term, the compactness analysis is more difficult than that of
Theorem \ref{T1.1}. To overcome this difficulty, it is crucial to estimate
the value of the mountain pass level with the help of the ground state $%
\widetilde{u}_{r}$ obtained by Theorem \ref{T1.4}.

The paper is structured as follows. In Section 2, we give the proof of
Theorem \ref{T1.1}. In Section 3, we prove Theorems \ref{T1.4} and \ref{T1.5}%
.

\section{The case of $q^{\ast }<q<p\leq 2^{\ast }$ and $\protect\beta =-1$}

In this section, we always assume that $N\geq 3, q^{\ast }<q<p\leq 2^{\ast
},\beta =-1$ and $\Omega $ is a bounded smooth domain.

\subsection{The high-energy solution}

First of all, we study the mountain pass level by using the monotonicity
trick \cite{BCJS,CJS}. We introduce a modified problem related to problem (%
\ref{e4}) as follows
\begin{equation}
\left\{
\begin{array}{ll}
-\Delta u+V(x)u+\lambda u=s|u|^{q-2}u-|u|^{p-2}u & \ \text{in}\ \Omega _{r},
\\
u\in H_{0}^{1}(\Omega _{r}),\ \int_{\Omega _{r}}|u|^{2}dx=a, &
\end{array}%
\right.   \label{e5}
\end{equation}%
where $\frac{1}{2}\leq s\leq 1$. Clearly, solutions of problem (\ref{e5})
correspond to critical points of the energy functional $I_{r,s}:S_{r,a}%
\rightarrow \mathbb{R}$ defined by
\begin{equation*}
I_{r,s}(u)=\frac{1}{2}\int_{\Omega _{r}}|\nabla u|^{2}dx+\frac{1}{2}%
\int_{\Omega _{r}}V(x)u^{2}dx-\frac{s}{q}\int_{\Omega _{r}}|u|^{q}dx+\frac{1%
}{p}\int_{\Omega _{r}}|u|^{p}dx
\end{equation*}%
on the constraint $S_{r,a}.$ Next, we investigate the geometry structure of
the energy functional $I_{r,s}$.

\begin{lemma}
\label{L2.1} Assume that condition $(V_{0})$ holds. For each $a>a_{V}$,
there exist $r_{a}>0$ and $u_{0,s},u_{1,s}\in S_{r_{a},a}$ such that

\begin{itemize}
\item[$(i)$] For any $r>r_{a}$ and $\frac{1}{2}\leq s\leq 1$, there hold $%
\Vert \nabla u_{0,s}\Vert _{2}^{2}<A_{a}<\Vert \nabla u_{1,s}\Vert _{2}^{2}$
and $I_{r,s}(u_{1,s})\leq 0<I_{r,s}(u_{0,s})<B_{a}$, where
\begin{equation*}
A_{a}:=\left[ \frac{2q}{N(q-2)C_{q}}\left( 1-\Vert V_{-}\Vert
_{N/2}S^{-1}\right) a^{\frac{q(N-2)-2N}{4}}\right] ^{\frac{4}{N(q-2)-4}}
\end{equation*}%
and
\begin{equation*}
B_{a}:=\frac{(N(q-2)-4)(1-\Vert V_{-}\Vert _{N/2}S^{-1})}{2N(q-2)}\left[
\frac{2q\left( 1-\Vert V_{-}\Vert _{N/2}S^{-1}\right) }{N(q-2)C_{q}}a^{-%
\frac{2q-N(q-2)}{4}}\right] ^{\frac{4}{N(q-2)-4}}.
\end{equation*}

\item[$(ii)$] Denote
\begin{equation*}
C_{r,s}(a)=\inf_{\gamma \in \Gamma _{r,a}}\sup_{t\in \lbrack
0,1]}I_{r,s}(\gamma (t))
\end{equation*}%
with $\Gamma _{r,a}:=\{\gamma \in C([0,1],S_{r,a}):\gamma (0)=u_{0,s},\
\gamma (1)=u_{1,s}\}$. Then we have $B_{a}\leq C_{r,s}(a)\leq H_{a}$, where $%
H_{a}:=\max\limits_{t\in (0,t_{\ast })}h(t)$ with the function $h:%
\mathbb{R}
^{+}\rightarrow
\mathbb{R}
$ being defined by%
\begin{equation*}
h(t)=\frac{1}{2}\left( 1+\Vert V\Vert _{N/2}S^{-1}\right) \theta at^{2}-%
\frac{1}{2q}a^{\frac{q}{2}}|\Omega |^{\frac{2-q}{2}}t^{\frac{N(q-2)}{2}}+%
\frac{C_{p}}{p}a^{\frac{p}{2}}\theta ^{\frac{N(p-2)}{4}}t^{\frac{N(p-2)}{2}}
\end{equation*}%
and
\begin{equation*}
t_{\ast }:=a^{-\frac{1}{N}}\left[ \frac{p(N(q-2)-4)}{2qC_{p}(N(p-2)-4)\theta
^{\frac{N(p-2)}{4}}|\Omega |^{\frac{q-2}{2}}}\right] ^{\frac{2}{N(p-q)}}.
\end{equation*}
\end{itemize}
\end{lemma}

\begin{proof}
$(i)$ Let $v_{1}\in S_{1,a}$ be the positive eigenfunction associated with $%
\theta $. Then by the H\"{o}lder inequality, we have
\begin{equation}
\int_{\Omega }|\nabla v_{1}|^{2}dx=\theta a\ \text{and}\ \int_{\Omega
}|v_{1}|^{q}dx\geq a^{\frac{q}{2}}|\Omega |^{\frac{2-q}{2}}.  \label{e6}
\end{equation}%
Set%
\begin{equation}
v_{t}(x):=t^{N/2}v_{1}(tx)\text{ for each }x\in \Omega _{1/t}\text{ and }t>0.
\label{e7}
\end{equation}%
Then it follows from (\ref{e6}), (\ref{e7}) and the Gagliardo-Nirenberg
inequality that%
\begin{eqnarray*}
I_{1/t,s}(v_{t}) &=&\frac{1}{2}\int_{\Omega _{1/t}}|\nabla v_{t}|^{2}dx+\frac{%
1}{2}\int_{\Omega _{1/t}}V(x)v_{t}^{2}dx-\frac{s}{q}\int_{\Omega
_{1/t}}|v_{t}|^{q}dx+\frac{1}{p}\int_{\Omega _{1/t}}|v_{t}|^{p}dx \\
&=&\frac{t^{2}}{2}\int_{\Omega }|\nabla v_{1}|^{2}dx+\frac{1}{2}\int_{\Omega
}V(x/t)|v_{1}|^{2}dx-\frac{st^{\frac{N(q-2)}{2}}}{q}\int_{\Omega
}|v_{1}|^{q}dx+\frac{t^{\frac{N(p-2)}{2}}}{p}\int_{\Omega }|v_{1}|^{p}dx \\
&\leq &\frac{t^{2}}{2}\left( 1+\Vert V\Vert _{N/2}S^{-1}\right) \int_{\Omega
}|\nabla v_{1}|^{2}dx-\frac{1}{2q}t^{\frac{N(q-2)}{2}}a^{\frac{q}{2}}|\Omega |^{\frac{2-q}{2}} \\
&&+\frac{t^{\frac{N(p-2)}{2}}C_{p}}{p}\left( \int_{\Omega
}|v_{1}|^{2}dx\right) ^{\frac{2N-(N-2)p}{4}}\left( \int_{\Omega }|\nabla
v_{1}|^{2}dx\right) ^{\frac{N(p-2)}{4}} \\
&=&\frac{1}{2}\left( 1+\Vert V\Vert _{N/2}S^{-1}\right) \theta at^{2}-\frac{1%
}{2q}a^{\frac{q}{2}}|\Omega |^{\frac{2-q}{2}}t^{\frac{N(q-2)}{2}}+\frac{C_{p}%
}{p}a^{\frac{p}{2}}\theta ^{\frac{N(p-2)}{4}}t^{\frac{N(p-2)}{2}} \\
&=&:h(t).
\end{eqnarray*}
According to $q^{\ast }<q<p\leq 2^{\ast }$ and the definition of $a_{V}$,
there exist $0<t_{1,a}<t_{2,a}$ such that $h(t)>0$ for any $t\in
(0,t_{1,a})\cup (t_{2,a},\infty )$, $h(t)<0$ for any $t\in (t_{1,a},t_{2,a})$
and $h(t_{1,a})=h(t_{2,a})=0$. Moreover, there hold $t_{1,a}<t_{\ast
}<t_{2,a}$ and $h(t_{\ast })<0$. Thus, we have
\begin{equation}
I_{r,s}(v_{1/t_{1,a}})=I_{1/t_{1},s}(v_{1/t_{1,a}})\leq h(t_{1,a})=0
\label{e8}
\end{equation}%
for any $r\geq \frac{1}{t_{1,a}}$ and $\frac{1}{2}\leq s\leq 1$. On the
other hand, by the Gagliardo-Nirenberg and the H\"{o}lder inequalities one
has
\begin{equation*}
I_{r,s}(u)\geq \frac{1}{2}\left( 1-\Vert V_{-}\Vert _{N/2}S^{-1}\right)
\int_{\Omega _{r}}|\nabla u|^{2}dx-\frac{C_{q}a^{\frac{2q-N(q-2)}{4}}}{q}%
\left( \int_{\Omega _{r}}|\nabla u|^{2}dx\right) ^{\frac{N(q-2)}{4}}.
\end{equation*}%
Let
\begin{equation*}
f(t):=\frac{1}{2}\left( 1-\Vert V_{-}\Vert _{N/2}S^{-1}\right) t-\frac{%
C_{q}a^{\frac{2q-N(q-2)}{4}}}{q}t^{\frac{N(q-2)}{4}}\text{ for }t>0.
\end{equation*}%
Then a direct calculation shows that $f$ is increasing on $t\in (0,t_{a})$
and decreasing on $(t_{a},\infty )$, where
\begin{equation*}
t_{a}:=\left[ \frac{2q}{C_{q}N(q-2)}(1-\Vert V_{-}\Vert _{N/2}S^{-1})a^{%
\frac{q(N-2)-2N}{4}}\right] ^{\frac{4}{N(q-2)-4}}.
\end{equation*}%
Moreover,
\begin{equation}
f(t_{a})=\max\limits_{t>0}f(t)=B_{a}>0.  \label{e9}
\end{equation}%
According to the geometric structure of $h(t)$ and (\ref{e8})-(\ref{e9}),
there exists a constant $0<t_{3,a}<t_{1,a}$ such that
\begin{equation}
h(t)<f(t_{a})\ \text{for all}\ t\in \lbrack 0,t_{3,a}].  \label{e10}
\end{equation}%
By (\ref{e9}) and (\ref{e10}), for each $r\geq r_{\ast }:=\max \left\{ \frac{1%
}{t_{3,a}},\sqrt{\frac{\theta a}{t_{a}}}\right\} $, we have $v_{1/r_{\ast
}}\in S_{r,a}$,
\begin{equation}
\Vert \nabla v_{1/r_{\ast }}\Vert _{2}^{2}=\left( \frac{1}{r_{\ast }}\right)
^{2}\Vert \nabla v_{1}\Vert _{2}^{2}<t_{a}  \label{e11}
\end{equation}%
and
\begin{equation}
0<I_{1/r_{\ast },s}(v_{1/r_{\ast }})\leq h(1/r_{\ast })\leq
h(t_{3,a})<f(t_{a}).  \label{e12}
\end{equation}%
Set $u_{0,s}:=v_{1/r_{\ast }}$, $u_{1,s}:=v_{t_{1,a}}$ and $r_{a}:=\max
\left\{ \frac{1}{t_{1,a}},r_{\ast }\right\} $. Thus, for any $r\geq r_{a},$
it follows from (\ref{e8}), (\ref{e11}) and (\ref{e12}) that $(i)$ holds.

$(ii)$ We define a path $\gamma :[0,1]\rightarrow S_{r,a}$ by
\begin{equation*}
\gamma (T):\Omega _{r}\rightarrow \mathbb{R},\ x\mapsto \left[ Tt_{\ast }+%
\frac{(1-T)}{r_{a}}\right] ^{\frac{N}{2}}v_{1}\left( \left( Tt_{\ast }+\frac{%
(1-T)}{r_{a}}\right) x\right) .
\end{equation*}%
It follows from (\ref{e9}) that $\max\limits_{t\in \lbrack
0,1]}I_{r,s}(\gamma (t))\geq f(t_{a})=B_{a}$. Moreover, according to the
definition of $\gamma ,$ we easily obtain that%
\begin{equation*}
\max\limits_{t\in \lbrack 0,1]}I_{r,s}(\gamma (t))=\max\limits_{t\in \lbrack
1/r_{a},t_{\ast }]}I_{r,s}(v_{t})=\max\limits_{t\in \lbrack 1/r_{a},t_{\ast
}]}I_{1/t,s}(v_{t})\leq \max_{t\in (0,t_{\ast })}h(t)=H_{a}.
\end{equation*}%
Hence $(ii)$ holds. The proof is complete.
\end{proof}

To obtain a bounded Palais-Smale sequence of the energy functional $I_{r,s}$%
, we recall the following theorem.

\begin{theorem}
\label{L2.2} (\cite[Theorem 1.5]{BCJS}) Let $(E,\langle \cdot ,\cdot \rangle )$
and $(H,\langle \cdot ,\cdot \rangle )$ be two infinite-dimensional Hilbert
spaces and assume there are continuous injections
\begin{equation*}
E\hookrightarrow H\hookrightarrow E^{\prime }.
\end{equation*}%
Let
\begin{equation*}
\Vert u\Vert _{2}^{2}=\langle u,u\rangle ,\ |u|^{2}=(u,u)\ \text{for}\ u\in E
\end{equation*}%
and
\begin{equation*}
S_{\mu }=\{u\in E:|u|^{2}=\mu \},\ T_{\mu }S_{\mu }=\{v\in E:(u,v)=0\}\
\text{for}\ \mu \in (0,+\infty ).
\end{equation*}%
Let $I\subset (0,+\infty )$ be an interval and consider a family of $C^{2}$
functional $\Phi _{\rho }:E\rightarrow \mathbb{R}$ of the form
\begin{equation*}
\Phi _{\rho }=A(u)-\rho B(u)\ \text{for}\ \rho \in I,
\end{equation*}%
with $B(u)\geq 0$ for every $u\in E$, and
\begin{equation*}
A(u)\rightarrow \infty \ \text{or}\ B(u)\rightarrow +\infty \ \text{as}\
u\in E\ \text{and}\ \Vert u\Vert \rightarrow +\infty .
\end{equation*}%
Moreover, $\Phi _{\rho }^{\prime }$ and $\Phi _{\rho }^{\prime \prime }$ are
$\tau $-H\"{o}lder continuous with $\tau \in (0,1]$, on bounded sets in the
following sense: for every $R>0$ there exists $M=M(R)>0$ such that
\begin{equation*}
\Vert \Phi _{\rho }^{\prime }(u)-\Phi _{\rho }^{\prime }(v)\Vert \leq M\Vert
u-v\Vert ^{\tau }\ \text{and}\ \Vert \Phi _{\rho }^{\prime \prime }(u)-\Phi
_{\rho }^{\prime \prime }(v)\Vert \leq M\Vert u-v\Vert ^{\tau }
\end{equation*}%
for every $u,v\in B(0,R)$. Finally, suppose that there exist $w_{1},w_{2}\in
S_{\mu }$ independent of $\rho $ such that
\begin{equation*}
c_{\rho }:=\inf\limits_{\gamma \in \Gamma }\max\limits_{t\in \lbrack
0,1]}\Phi _{\rho }(\gamma (t))>\max \{\Phi _{\rho }(w_{1}),\Phi _{\rho
}(w_{2})\}\ \text{for all}\ \rho \in I,
\end{equation*}%
where
\begin{equation*}
\Gamma :=\{\gamma \in C([0,1],S_{\mu }):\gamma (0)=w_{1},\ \gamma
(1)=w_{2}\}.
\end{equation*}%
Then for almost every $\rho \in I,$ there exists a sequence $%
\{u_{n}\}\subset S_{\mu }$ such that\newline
$(i)$ $\Phi _{\rho }(u_{n})\rightarrow c_{\rho };$ $(ii)$ $\Phi _{\rho
}^{\prime }|_{S_{\mu }}(u_{n})\rightarrow 0;$ $(iii)$ $\{u_{n}\}$ is bounded
in $E.$
\end{theorem}

According to Lemma \ref{L2.1} and Theorem \ref{L2.2}, we have the following
results.

\begin{theorem}
\label{L2.3} Assume that condition $(V_{0})$ holds. Then for any $a>a_{V}$
and $r>r_{a}$, problem (\ref{e5}) has a solution $(\lambda _{r,s},u_{r,s})$
for almost every $\frac{1}{2}\leq s\leq 1$. Moreover, there hold $%
u_{r,s}\geq 0$ and $I_{r,s}(u_{r,s})=C_{r,s}(a)$.
\end{theorem}

\begin{proof}
Similar to the argument of \cite[Theorem 2.4]{BQZ}, by Lemma \ref{L2.1} and
Theorem \ref{L2.2}, for almost every $\frac{1}{2}\leq s\leq 1$, there exists
a nonnegative bounded Palais-Smale sequence $\{u_{n}\}\subset S_{r,a}$ such
that%
\begin{equation}
I_{r,s}(u_{n})\rightarrow C_{r,s}(a)\ \text{and}\ I_{r,s}^{\prime
}(u_{n})|_{T_{u_{n}}S_{r,a}}\rightarrow 0\text{ as }n\rightarrow \infty ,
\label{e13}
\end{equation}%
where $T_{u_{n}}S_{r,a}$ denote the tangent space of $S_{r,a}$ at $u_{n}$.
It follows from (\ref{e13}) that there exists $\{\lambda _{n}\}\subset
\mathbb{R}$ such that
\begin{equation}
I_{r,s}^{\prime }(u_{n})+\lambda _{n}u_{n}\rightarrow 0\ \text{in}\
H^{-1}(\Omega _{r})\text{ as }n\rightarrow \infty   \label{e14}
\end{equation}%
and
\begin{equation*}
\lambda _{n}:=-\frac{2}{a}C_{r,s}(a)+\frac{(q-2)s}{qa}\int_{\Omega
_{r}}|u_{n}|^{q}dx+\frac{p-2}{pa}\int_{\Omega _{r}}|u_{n}|^{p}dx+o_{n}(1).
\end{equation*}%
Clearly, $\{\lambda _{n}\}$ is bounded. Thus, there exist $u_{r,s}\in
H_{0}^{1}(\Omega _{r})$ and $\lambda _{r,s}\in \mathbb{R}$ such that up to a
subsequence,
\begin{equation}
\begin{array}{l}
\lambda _{n}\rightarrow \lambda _{r,s}, \\
u_{n}\rightharpoonup u_{r,s}\ \text{in}\ H_{0}^{1}(\Omega _{r}), \\
u_{n}\rightarrow u_{r,s}\ \text{in}\ L^{k}(\Omega _{r})\ \text{for all}\
2\leq k<2^{\ast },%
\end{array}
\label{e15}
\end{equation}%
and
\begin{equation}
\int_{\Omega _{r}}|\nabla u_{r,s}|^{2}dx+\int_{\Omega
_{r}}V(x)u_{r,s}^{2}dx-s\int_{\Omega _{r}}|u_{r,s}|^{q}dx+\int_{\Omega
_{r}}|u_{r,s}|^{p}dx+\lambda _{r,s}\int_{\Omega _{r}}u_{r,s}^{2}dx=0.
\label{e16}
\end{equation}%
Combining with (\ref{e14})-(\ref{e16}), we deduce that
\begin{eqnarray*}
o_{n}(1) &=&\int_{\Omega _{r}}|\nabla u_{n}|^{2}dx+\int_{\Omega
_{r}}V(x)u_{n}^{2}dx-\int_{\Omega _{r}}|u_{n}|^{q}dx+\int_{\Omega
_{r}}|u_{n}|^{p}dx+\lambda _{n}\int_{\Omega _{r}}u_{n}^{2}dx \\
&&+\frac{q-2}{2q}\int_{\mathbb{R}^{N}}V(x)u_{r,s}^{2}dx-\frac{p-q}{pq}\int_{%
\mathbb{R}^{N}}|u_{r,s}|^{p}dx \\
&\leq &\frac{(N-2)q-2N-q\Vert \widetilde{V}_{+}\Vert
_{N/2}S^{-1}-(q-2)N\Vert V_{+}\Vert _{N/2}}{2NqS}\int_{\mathbb{R}%
^{N}}|\nabla u_{r,s}|^{2}dx \\
&<&0.
\end{eqnarray*}%
which shows that $u_{n}\rightarrow u_{r,s}\geq 0$ in $H_{0}^{1}(\Omega _{r})$. Hence, $(\lambda _{r,s},u_{r,s})$ is a solution of problem (\ref{e5}) and $%
I_{r,s}(u_{r,s})=C_{r,s}(a).$ The proof is complete.
\end{proof}

Before proving a mountain pass solution to problem (\ref{e4}), we present a
consistent estimate of the solution to problem (\ref{e5}).

\begin{lemma}
\label{L2.4} Assume that condition $(V_{0})$ holds. If $(\lambda
_{r,s},u_{r,s})\in \mathbb{R}\times S_{r,a}$ is a solution to problem (\ref%
{e5}), then $u_{r,s}$ is bounded uniformly in $s$ and $r.$
\end{lemma}

\begin{proof}
Let $(\lambda _{r,s},u_{r,s})\in \mathbb{R}\times S_{r,a}$ be a solution of
problem (\ref{e5}) for almost every $\frac{1}{2}\leq s\leq 1$. A direct
calculation shows that
\begin{eqnarray}
&& \frac{1}{2}(1-\Vert V_{-}\Vert _{N/2}S^{-1})\int_{\Omega _{r}}|\nabla
u_{r,s}|^{2}dx  \notag \\
&\leq& \frac{1}{2}\int_{\Omega _{r}}|\nabla u_{r,s}|^{2}dx+\frac{1}{2}%
\int_{\Omega _{r}}V(x)u_{r,s}^{2}dx  \notag \\
&=&\frac{s}{q}\int_{\Omega _{r}}|u_{r,s}|^{q}dx-\frac{1}{p}\int_{\Omega
_{r}}|u_{r,s}|^{p}dx+C_{r,s}(a)  \notag \\
&\leq& \frac{a^{\frac{p-q}{p-2}}}{q}\left( \int_{\Omega
_{r}}|u_{r,s}|^{p}dx\right) ^{\frac{q-2}{p-2}}-\frac{1}{p}\int_{\Omega
_{r}}|u_{r,s}|^{p}dx+C_{r,s}(a).  \label{e17}
\end{eqnarray}
Set
\begin{equation*}
g(t):=\frac{a^{\frac{p-q}{p-2}}}{q}t^{\frac{q-2}{p-2}}-\frac{1}{p}t\text{
for }t>0.
\end{equation*}%
Then we easily get that $g(t)>0$ for $t\in (0,t_{4,a})$ and $g(t)<0$ for $%
t\in (t_{4,a},\infty )$, where $t_{4,a}:=a\left( \frac{p}{q}\right) ^{\frac{%
p-2}{p-q}}$. Moreover, there holds
\begin{equation}
\max\limits_{t\geq 0}g(t)=g(t^{\ast })=\frac{a(p-q)}{q(p-2)}\left( \frac{%
p(q-2)}{q(p-2)}\right) ^{\frac{q-2}{p-q}},  \label{e18}
\end{equation}%
where
\begin{equation*}
t^{\ast }:=a\left( \frac{p(q-2)}{q(p-2)}\right) ^{\frac{p-2}{p-q}}.
\end{equation*}%
Combining with (\ref{e17}) and (\ref{e18}), we have
\begin{equation*}
\frac{1}{2}(1-\Vert V_{-}\Vert _{N/2}S^{-1})\int_{\Omega _{r}}|\nabla
u_{r,s}|^{2}dx\leq \frac{a(p-q)}{q(p-2)}\left( \frac{p(q-2)}{q(p-2)}\right)
^{\frac{q-2}{p-q}}+C_{r,s}(a),
\end{equation*}%
which implies that $\{u_{r,s}\}$ is bounded.
\end{proof}

According to Theorem \ref{L2.3} and Lemma \ref{L2.4}, we have the following
existence result.

\begin{theorem}
\label{L2.5} For each $a>a_{V}$ and $r>r_{a}$, problem (\ref{e4}) admits a solution $u_{r}$ satisfying $u_{r}\geq 0$ in $\Omega _{r}$ and $I_{r}(u_{r})>0$ for some Lagrange
multiplier $\lambda _{r}\in \mathbb{R}$.
\end{theorem}

\begin{proof}
By Theorem \ref{L2.3}, there is a nonnegative solution $(\lambda
_{r,s},u_{r,s})$ to problem (\ref{e5}) for almost every $\frac{1}{2}\leq
s\leq 1$. Moreover, $\{u_{r,s}\}$ is bounded uniformly in $s$ and $r$ by
Lemma \ref{L2.4}. So, similar to the argument in Theorem \ref{L2.3}, there
exist $u_{r}\in S_{r,a}$ and $\lambda _{r}\in \mathbb{R}$ such that up to a
subsequence,
\begin{equation*}
\lambda _{r,s}\rightarrow \lambda _{r}\ \text{and}\ u_{r,s}\rightarrow
u_{r}\ \text{in}\ H_{0}^{1}(\Omega _{r})\ \text{as}\ s\rightarrow 1.
\end{equation*}%
Moreover, $u_{r}$ is a non-negative solution to problem (\ref{e4}). The
proof is complete.
\end{proof}

\begin{lemma}
\label{L2.6} Let $\{(\lambda _{r},u_{r})\}$ be a family of nonnegative
solutions of problem (\ref{e4}) satisfying $\Vert u_{r}\Vert _{H^{1}}\leq C $%
, where $C>0$ is independent of $r$. Then there holds $\limsup\limits_{r%
\rightarrow \infty }\Vert u_{r}\Vert _{\infty }<\infty $.
\end{lemma}

\begin{proof}
According to the regularity theory of elliptic partial differentical
equations, we have $u_{r}\in C(\Omega _{r})$. Suppose to the contrary, then
there exist a sequence, denoted by $\{u_{r}\}$ for simplicity, and $x_{r}\in
\Omega _{r}$ such that
\begin{equation*}
M_{r}:=\max\limits_{x\in \Omega _{r}}u_{r}(x)=u_{r}(x_{r})\rightarrow \infty
\ \text{as}\ r\rightarrow \infty .
\end{equation*}%
Without loss of generality, we assume that up to a subsequence, $%
\lim\limits_{r\rightarrow \infty }\frac{x_{r}}{|x_{r}|}=(1,0,\cdots ,0)$.
Denote
\begin{equation*}
v_{r}(x)=\frac{u_{r}(x_{r}+\tau _{r}x)}{M_{r}}\ \text{for}\ x\in \Gamma
_{r}:=\{x\in \mathbb{R}^{N}:x_{r}+\tau _{r}x\in \Omega _{r}\},
\end{equation*}%
where $\tau _{r}=M_{r}^{\frac{2-p}{2}}$. Clearly, $\tau _{r}\rightarrow 0$, $%
\Vert v_{r}\Vert _{\infty ,\Gamma _{r}}\leq 1$ and
\begin{equation}
\begin{array}{ll}
-\Delta v_{r}+\tau _{r}^{2}V(x_{r}+\tau _{r}x)v_{r}+\tau _{r}^{2}\lambda
_{r}v_{r}=|v_{r}|^{q-2}v_{r}-\tau _{r}^{\frac{2(p-q)}{q-2}}|v_{r}|^{p-2}v_{r}
& \quad \text{in}\ \Gamma _{r}.%
\end{array}
\label{e19}
\end{equation}%
Since $\Vert u_{r}\Vert _{H^{1}}\leq C$ with $C$ independent of $r$, it
follows from problem (\ref{e4}) that the sequence $\{\lambda _{r}\}$ is
bounded. Combining with the regularity theory of elliptic partial
differential equations and the Arze-Ascoli theorem, we obtain that there
exists $v$ such that up to a subsequence
\begin{equation*}
v_{r}\rightharpoonup v\ \text{in}\ H_{0}^{1}\left( \Gamma \right) \ \text{and%
}\ v_{r}\rightarrow v\ \text{in}\ C_{loc}^{\beta }\left( \Gamma \right) \
\text{for some}\ \beta \in (0,1),
\end{equation*}%
where $\Gamma :=\lim\limits_{r\rightarrow \infty }\Gamma _{r}$. Similar to
the argument in \cite[Lemma 2.7]{BQZ}, we have
\begin{equation}
\liminf\limits_{r\rightarrow \infty }\frac{dist(x_{r},\partial \Omega _{r})}{%
\tau _{r}}=\liminf\limits_{r\rightarrow \infty }\frac{|y_{r}-x_{r}|}{\tau
_{r}}\geq d>0,  \label{e20}
\end{equation}%
where $y_{r}\in \partial \Omega _{r}$ is such that $dist(x_{r},\partial
\Omega _{r})=|y_{r}-x_{r}|$ for any large $r$. So, it follows from (\ref{e19}%
) and (\ref{e20}) that $v\in H_{0}^{1}(\Gamma )$ is a nonnegative solution
of
\begin{equation*}
\begin{array}{ll}
-\Delta v=|v|^{q-2}v & \ \text{in}\ \Gamma ,%
\end{array}%
\end{equation*}%
where
\begin{equation*}
\Gamma :=%
\begin{cases}
\mathbb{R}^{N},\  & \text{if}\ \liminf\limits_{r\rightarrow \infty }\frac{%
dist(x_{r},\partial \Omega _{r})}{\tau _{r}}=\infty , \\
\{x\in \mathbb{R}^{N}:x_{1}>-d\},\  & \text{if}\
\liminf\limits_{r\rightarrow \infty }\frac{dist(x_{r},\partial \Omega _{r})}{%
\tau _{r}}>0.%
\end{cases}%
\end{equation*}%
Using Liouville's theorem \cite{EL}, we derive $v=0$ in $H_{0}^{1}(\Gamma )$%
, which contradicts with $v(0)=\lim\limits_{r\rightarrow \infty }v_{r}(0)=1$%
. The proof is complete.
\end{proof}

\subsection{The global minimizer}

\begin{lemma}
\label{L2.7} Assume that condition $(V_{0})$ holds. Then the energy
functional $I_{r}$ is bounded from below on $S_{r,a}$ for all $a>0.$
\end{lemma}

\begin{proof}
For any $u\in S_{r,a}$, by condition $(V_{0})$ and the H\"{o}lder inequality
we have
\begin{eqnarray*}
I_{r}(u)& \geq& \frac{1}{2}(1-\Vert V_{-}\Vert _{N/2}S^{-1})\int_{\Omega
_{r}}|\nabla u|^{2}dx-\frac{1}{q}\int_{\Omega _{r}}|u|^{q}dx+\frac{1}{p}%
\int_{\Omega _{r}}|u|^{p}dx \\
& \geq& -\frac{1}{q}\int_{\Omega _{r}}|u|^{q}dx+\frac{1}{p}\int_{\Omega
_{r}}|u|^{p}dx \\
& \geq &-\frac{a^{\frac{p-q}{p-2}}}{q}\left( \int_{\Omega
_{r}}|u|^{p}dx\right) ^{\frac{q-2}{p-2}}+\frac{1}{p}\int_{\Omega
_{r}}|u|^{p}dx \\
& \geq& -\frac{a(p-q)}{q(p-2)}\left[ \frac{p(q-2)}{q(p-2)}\right] ^{\frac{q-2%
}{p-q}},
\end{eqnarray*}%
since $q^{\ast }<q<p\leq 2^{\ast }.$ This shows that $I_{r}$ is bounded from
below on $S_{r,a}$ for all $a>0$. The proof is complete.
\end{proof}

Define
\begin{equation*}
c_{r,a}:=\inf\limits_{u\in S_{r,a}}I_{r}(u).
\end{equation*}%
Clearly, by Lemma \ref{L2.7} one has%
\begin{equation*}
c_{r,a}\geq -\frac{a(p-q)}{q(p-2)}\left[ \frac{p(q-2)}{q(p-2)}\right] ^{%
\frac{q-2}{p-q}}.
\end{equation*}%
Furthermore, we have the following results.

\begin{theorem}
\label{L2.8} Assume that condition $(V_{0})$ holds. Let $a>a_{V}$. If $r>%
\overline{r}_{a}:=\frac{1}{t_{\ast }}$, then
\begin{equation*}
-\frac{a(p-q)}{q(p-2)}\left( \frac{p(q-2)}{q(p-2)}\right) ^{\frac{q-2}{p-q}%
}\leq c_{r,a}<0
\end{equation*}%
is achieved by $\overline{u}_{r}\in S_{r,a}$, where $t_{\ast }$ is given in
Lemma \ref{L2.1}. In other words, there exists a Lagrange multiplier $%
\overline{\lambda }_{r}\in \mathbb{R}$ such that $(\overline{\lambda }_{r},%
\overline{u}_{r})$ is a solution of problem (\ref{e4}). Moreover, there holds%
\begin{equation*}
\Vert \nabla \overline{u}_{r}\Vert _{2}\geq \left[ \frac{2q}{N(q-2)C_{q}}%
\left( 1-\Vert V_{-}\Vert _{N/2}S^{-1}\right) a^{\frac{q(N-2)-2N}{4}}\right]
^{\frac{2}{N(q-2)-4}}.
\end{equation*}
\end{theorem}

\begin{proof}
According to the proof of Lemma \ref{L2.1}, for any $r>\frac{1}{t_{\ast }}$,
there holds
\begin{equation*}
c_{r,a}=\inf_{u\in S_{r,a}}I_{r}(u)\leq I_{r}(v_{1/t_{\ast }})<0.
\end{equation*}%
It follows from the Ekeland variational principle \cite{E} that there exists
a sequence $\{u_{n}\}\subset S_{r,a}$ such that
\begin{equation*}
I_{r}(u_{n})\rightarrow c_{r,a}\ \text{and }I_{r}^{\prime
}(u_{n})|_{T_{u_{n,r}}S_{r,a}}\rightarrow 0\ \text{as}\ n\rightarrow \infty .
\end{equation*}%
Similar to the proof of Lemma \ref{L2.4}, we obtain that $\{u_{n}\}$ is
bounded in $H_{0}^{1}(\Omega _{r}),$ since $c_{r,a}<0$. Then there exists $%
\overline{u}_{r}\in H_{0}^{1}(\Omega _{r})$ such that up to a subsequence,
\begin{equation}
u_{n}\rightharpoonup \overline{u}_{r}\ \text{in}\ H_{0}^{1}(\Omega _{r})\
\text{and}\ u_{n}\rightarrow \overline{u}_{r}\ \text{in}\ L^{k}(\Omega
_{r})\ \text{for all}\ 2\leq k<2^{\ast }.  \label{e21}
\end{equation}%
Clearly, $\overline{u}_{r}\in S_{r,a}$. Moreover, there holds%
\begin{equation*}
\Vert \nabla \overline{u}_{r}\Vert _{2}^{2}\leq \liminf\limits_{n\rightarrow
\infty }\Vert \nabla u_{n}\Vert _{2}^{2}.
\end{equation*}%
It follows from condition $(V_{0})$ and (\ref{e21}) that
\begin{eqnarray*}
c_{r,a}\leq I_{r}(\overline{u}_{r})&\leq & \liminf_{n\rightarrow \infty
}I_{r}(u_{n}) \\
&=& I_{r}(\overline{u}_{r})+\liminf_{n\rightarrow \infty }\left( \frac{1}{2}%
\int_{\Omega _{r}}|\nabla (u_{n}-\overline{u}_{r})|^{2}dx+\frac{1}{p}%
\int_{\Omega _{r}}|u_{n}-\overline{u}_{r}|^{p}dx\right)  \\
&=& c_{r,a}.
\end{eqnarray*}%
Hence, we have $u_{n}\rightarrow \overline{u}_{r}$ in $H_{0}^{1}(\Omega _{r})
$ and $I_{r}(\overline{u}_{r})=c_{r,a}$. The lagrange multiplier theorem
implies that $(\overline{\lambda }_{r},\overline{u}_{r})$ is a solution of
problem (\ref{e4}). Finally, by the strong maximum principle one has $%
\overline{u}_{r}>0$. Furthermore, according to the argument of Lemma \ref%
{L2.1}, we deduce that%
\begin{equation*}
\Vert \nabla \overline{u}_{r}\Vert _{2}\geq \left[ \frac{2q}{N(q-2)C_{q}}%
\left( 1-\Vert V_{-}\Vert _{N/2}S^{-1}\right) a^{\frac{q(N-2)-2N}{4}}\right]
^{\frac{2}{N(q-2)-4}}.
\end{equation*}%
The proof is complete.
\end{proof}

\begin{lemma}
\label{L2.9} Let $(\overline{\lambda }_{r},\overline{u}_{r})$ be a solution
of problem (\ref{e4}) obtained by Theorem \ref{L2.8}. If in addition $V\in
C^{1}(\mathbb{R}^{N})$ satisfies
\begin{equation}
q\Vert \widetilde{V}_{+}\Vert _{N/2}+N(q-2)\Vert V_{+}\Vert
_{N/2}<(2N-(N-2)q)S,  \label{e22}
\end{equation}%
then for any $a>a_{V}$, there holds
\begin{equation*}
\liminf_{r\rightarrow \infty }\overline{\lambda }_{r}>0.
\end{equation*}
\end{lemma}

\begin{proof}
It follows from Theorem \ref{L2.4} that $\{\overline{u}_{r}:r>\overline{r}%
_{a}\}$ is bounded in $H^{1}(\mathbb{R}^{N})$. Then there exist $\overline{u}%
_{\infty }\in H^{1}(\mathbb{R}^{N})$ and $\overline{\lambda }_{\infty }\in
\mathbb{R}$ such that up to a subsequence,
\begin{align*}
\overline{\lambda }_{r}& \rightarrow \overline{\lambda }_{\infty }, \\
\overline{u}_{r}& \rightharpoonup \overline{u}_{\infty }\ \text{in}\ H^{1}(%
\mathbb{R}^{N}), \\
\overline{u}_{r}& \rightarrow \overline{u}_{\infty }\ \text{in}\ L_{loc}^{k}(%
\mathbb{R}^{N})\ \text{for all}\ 2\leq k<2^{\ast }, \\
\overline{u}_{r}& \rightarrow \overline{u}_{\infty }\ \text{a.e. in}\
\mathbb{R}^{N}.
\end{align*}%
We now discuss it into two cases.

Case $(i):\overline{u}_{\infty }=0$. We first claim that there exists $d>0$
for any $n\in
\mathbb{N}
$ such that
\begin{equation}
\liminf_{r\rightarrow \infty }\sup_{z\in \mathbb{R}^{N}}\int_{B(z,1)}|%
\overline{u}_{r}|^{2}dx\geq d.  \label{e23}
\end{equation}%
Otherwise, by the concentration compactness principle one has
\begin{equation*}
\overline{u}_{r}\rightarrow 0\ \text{in}\ L^{k}(\mathbb{R}^{N})\ \text{for
all}\ 2<k<2^{\ast }\text{ as}\ r\rightarrow \infty .
\end{equation*}%
Then there holds
\begin{equation*}
0>I_{r}(\overline{u}_{r})=\frac{1}{2}\int_{\Omega _{r}}|\nabla \overline{u}%
_{r}|^{2}dx+\frac{1}{p}\int_{\Omega _{r}}|\overline{u}_{r}|^{p}dx+o_{r}(1),
\end{equation*}%
which shows that $\overline{u}_{r}\rightarrow 0$ in $H^{1}(\mathbb{R}^{N}).$
This is impossible. So, (\ref{e23}) holds, and there is $\overline{x}_{r}\in
\Omega _{r}$ such that
\begin{equation*}
\int_{B(\overline{x}_{r},1)}|\overline{u}_{r}|^{2}dx\geq \frac{d}{2}.
\end{equation*}%
Then $|\overline{x}_{r}|\rightarrow \infty $ as $r\rightarrow \infty $.
Otherwise, there is $R>0$ such that $|\overline{x}_{r}|<R$. Then we have
\begin{equation*}
\int_{B(0,R+1)}|\overline{u}_{r}|^{2}dx\geq \int_{B(\overline{x}_{r},1)}|%
\overline{u}_{r}|^{2}dx\geq \frac{d}{2}.
\end{equation*}%
This contradicts with the fact of $\overline{u}_{r}\rightharpoonup 0$ in$\
H^{1}(\mathbb{R}^{N})$ as $r\rightarrow \infty $.

Let $\overline{w}_{r}(x)=\overline{u}_{r}(x+\overline{x}_{r})$ for any $x\in
\overline{\Gamma }_{r}:=\{x\in \mathbb{R}^{N}:x+\overline{x}_{r}\in \Omega
_{r}\}$. Obviously, $\overline{w}_{r}$ is bounded in $H^{1}(\mathbb{R}^{N})$%
. Up to a subsequence, there is $\overline{w}_{\infty }\in H^{1}(\mathbb{R}%
^{N})$ such that $\overline{w}_{r}\rightharpoonup \overline{w}_{\infty }$ as
$r\rightarrow \infty $. Now we claim that $dist(\overline{x}_{r},\partial
\Omega _{r})\rightarrow \infty $ as $r\rightarrow \infty $. If not, we
assume that $\lim\limits_{r\rightarrow \infty }dist(\overline{x}%
_{r},\partial \Omega _{r})=M\in \lbrack 0,\infty )$. Without loss of the
generality, we assume that, up to a subsequence,
\begin{equation*}
\lim_{r\rightarrow \infty }\frac{\overline{x}_{r}}{|\overline{x}_{r}|}=e_{1}.
\end{equation*}%
Denote
\begin{equation*}
\Gamma _{1}:=\{x\in \mathbb{R}^{N}:x\cdot e_{1}<M\}=\{x\in \mathbb{R}%
^{N}:x_{1}<M\}.
\end{equation*}%
We have $\phi (\cdot -\overline{x}_{r})\in C_{c}^{\infty }(\Omega _{r})$ for
any $\phi \in C_{c}^{\infty }(\Gamma _{1})$ and $r$ large enough. Combining
with problem (\ref{e4}), we get
\begin{eqnarray*}
&& \int_{\Omega _{r}}\nabla \overline{u}_{r}\nabla \phi (\cdot -\overline{x}%
_{r})dx+\int_{\Omega _{r}}V(x)\overline{u}_{r}\phi (\cdot -\overline{x}%
_{r})dx+\overline{\lambda }_{r}\int_{\Omega _{r}}\overline{u}_{r}\phi (\cdot
-\overline{x}_{r})dx \\
& =&\int_{\Omega _{r}}|\overline{u}_{r}|^{q-2}\overline{u}_{r}\phi (\cdot -%
\overline{x}_{r})dx-\int_{\Omega _{r}}|\overline{u}_{r}|^{p-2}\overline{u}%
_{r}\phi (\cdot -\overline{x}_{r})dx.
\end{eqnarray*}%
A direction calculation shows that
\begin{equation}
\int_{\Omega _{r}}\nabla \overline{w}_{r}\nabla \phi dx+\int_{\Omega
_{r}}V(\cdot +\overline{x}_{r})\overline{w}_{r}\phi dx+\overline{\lambda }%
_{r}\int_{\Omega _{r}}\overline{w}_{r}\phi dx=\int_{\Omega _{r}}|\overline{w}%
_{r}|^{q-2}\overline{w}_{r}\phi dx-\int_{\Omega _{r}}|\overline{w}_{r}|^{p-2}%
\overline{w}_{r}\phi dx.  \label{e24}
\end{equation}%
Since $|\overline{x}_{r}|\rightarrow \infty $ as $r\rightarrow \infty $, we
have
\begin{eqnarray*}
\left\vert \int_{\Omega _{r}}V(\cdot +\overline{x}_{r})\overline{w}_{r}\phi
dx\right\vert & \leq& \Vert \overline{w}_{r}\Vert _{2^{\ast }}\Vert \phi
\Vert _{2^{\ast }}\left( \int_{supp\phi }|V(\cdot +\overline{x}%
_{r})|^{N/2}dx\right) ^{2/N} \\
& \leq& \Vert \overline{w}_{r}\Vert _{2^{\ast }}\Vert \phi \Vert _{2^{\ast
}}\left( \int_{\mathbb{R}^{N}\setminus \mathbf{B}_{|\overline{x}%
_{r}|/2}}|V(\cdot +\overline{x}_{r})|^{N/2}dx\right) ^{N/2} \\
& \rightarrow& 0\ \text{as}\ r\rightarrow \infty .
\end{eqnarray*}%
Then it follows from (\ref{e24}) that for any $\phi \in C_{c}^{\infty
}(\Gamma _{1})$,
\begin{equation*}
\int_{\Gamma _{1}}\nabla \overline{w}_{\infty }\nabla \phi dx+\overline{%
\lambda }_{\infty }\int_{\Gamma _{1}}\overline{w}_{\infty }\phi
dx=\int_{\Gamma _{1}}|\overline{w}_{\infty }|^{q-2}\overline{w}_{\infty
}\phi dx-\int_{\Gamma _{1}}|\overline{w}_{\infty }|^{p-2}\overline{w}%
_{\infty }\phi dx,
\end{equation*}%
from which we know that $\overline{w}_{\infty }\in H_{0}^{1}(\Gamma _{1})$
is a weak solution of the equation
\begin{equation}
\begin{array}{ll}
-\Delta u+\overline{\lambda }_{\infty }u=|u|^{q-2}u-|u|^{q-2}u & \ \text{in}%
\ \Gamma _{1},%
\end{array}
\label{e25}
\end{equation}%
which is impossible, since Eq. (\ref{e25}) has no nontrivial solution on a
half space by \cite{EL}. This shows that $dist(\overline{x}_{r},\partial
\Omega _{r})\rightarrow \infty $ as $r\rightarrow \infty ,$ which implies
that $\Gamma _{1}=\mathbb{R}^{N}$. Thus, $\overline{w}_{\infty }\in H^{1}(%
\mathbb{R}^{N})$ is a weak solution of the equation
\begin{equation*}
\begin{array}{ll}
-\Delta u+\overline{\lambda }_{\infty }u=|u|^{q-2}u-|u|^{q-2}u & \ \text{in}%
\ \mathbb{R}^{N}.%
\end{array}%
\end{equation*}%
Using the Pohozaev identity, we conclude that $\overline{\lambda }_{\infty
}>0$.

Case $(ii):\overline{u}_{\infty }\neq 0$. Note that $\overline{u}_{\infty }$
is a solution of the equation
\begin{equation*}
\begin{array}{ll}
-\Delta u+V(x)u+\overline{\lambda }_{\infty }u=|u|^{q-2}u-|u|^{p-2}u & \
\text{in}\ \mathbb{R}^{N}.%
\end{array}%
\end{equation*}%
Then, we get
\begin{equation}
\int_{\mathbb{R}^{N}}|\nabla \overline{u}_{\infty }|^{2}dx+\int_{\mathbb{R}%
^{N}}V(x)\overline{u}_{\infty }^{2}dx+\overline{\lambda }_{\infty }\int_{%
\mathbb{R}^{N}}\overline{u}_{\infty }^{2}dx=\int_{\mathbb{R}^{N}}|\overline{u%
}_{\infty }|^{q}dx-\int_{\mathbb{R}^{N}}|\overline{u}_{\infty }|^{p}dx
\label{e26}
\end{equation}%
and the Pohozave identity
\begin{eqnarray}
&& \frac{N-2}{2N}\int_{\mathbb{R}^{N}}|\nabla \overline{u}_{\infty }|^{2}dx+%
\frac{1}{2N}\int_{\mathbb{R}^{N}}\widetilde{V}(x)\overline{u}_{\infty
}^{2}dx+\frac{1}{2}\int_{\mathbb{R}^{N}}V(x)\overline{u}_{\infty }^{2}dx+%
\frac{\overline{\lambda }_{\infty }}{2}\int_{\mathbb{R}^{N}}\overline{u}%
_{\infty }^{2}dx  \notag \\
& =&\frac{1}{q}\int_{\mathbb{R}^{N}}|\overline{u}_{\infty }|^{q}dx-\frac{1}{p}%
\int_{\mathbb{R}^{N}}|\overline{u}_{\infty }|^{p}dx.  \label{e27}
\end{eqnarray}%
It follows from (\ref{e22}) and (\ref{e26})-(\ref{e27}) that
\begin{eqnarray*}
\frac{(2-q)\overline{\lambda }_{\infty }}{2q}\int_{\mathbb{R}^{N}}\overline{u%
}_{\infty }^{2}dx &=&\frac{(N-2)q-2N}{2Nq}\int_{\mathbb{R}^{N}}|\nabla
\overline{u}_{\infty }|^{2}dx+\frac{1}{2N}\int_{\mathbb{R}^{N}}\widetilde{V}%
(x)\overline{u}_{\infty }^{2}dx \\
&&+\frac{q-2}{2q}\int_{\mathbb{R}^{N}}V(x)\overline{u}_{\infty }^{2}dx-\frac{%
p-q}{pq}\int_{\mathbb{R}^{N}}|\overline{u}_{\infty }|^{p}dx \\
&\leq &\frac{(N-2)q-2N-q\Vert \widetilde{V}_{+}\Vert
_{N/2}S^{-1}-(q-2)N\Vert V_{+}\Vert _{N/2}}{2NqS}\int_{\mathbb{R}%
^{N}}|\nabla \overline{u}_{\infty }|^{2}dx \\
&<&0.
\end{eqnarray*}%
Therefore, if $\overline{u}_{\infty }\neq 0$ for $a>a_{V}$, then we have $%
\overline{\lambda }_{\infty }>0$. The proof is complete.
\end{proof}

\textbf{We are ready to prove Theorem \ref{T1.1}:} It is a direct
consequence of Theorems \ref{L2.5} and \ref{L2.8} and Lemmas \ref{L2.6} and %
\ref{L2.9}.

\section{The case of $2<q<q^{\ast }<p=2^{\ast }$ and $\protect\beta =1$}

In this section, we always assume that $N\geq 3,2<q<q^{\ast }<p=2^{\ast }$
and $\beta =1.$ First of all, we find a ground state of problem (%
\ref{e4}).

\subsection{The ground state solution}

It follows from condition $(V_{0}),$ the Gagliardo-Nirenberg and H\"{o}lder
inequalities that
\begin{eqnarray}
I_{r}(u) &=&\frac{1}{2}\int_{\Omega _{r}}|\nabla u|^{2}dx+\frac{1}{2}%
\int_{\Omega _{r}}V(x)u^{2}dx-\frac{1}{q}\int_{\Omega _{r}}|u|^{q}dx-\frac{1%
}{2^{\ast }}\int_{\Omega _{r}}|u|^{2^{\ast }}dx  \notag \\
&\geq &\frac{1}{2}(1-\Vert V_{-}\Vert _{N/2}S^{-1})\int_{\Omega _{r}}|\nabla
u|^{2}dx-\frac{C_{q}a^{\frac{2q-N(q-2)}{4}}}{q}\left( \int_{\Omega
_{r}}|\nabla u|^{2}dx\right) ^{\frac{N(q-2)}{4}}  \notag \\
&&-\frac{1}{2^{\ast }S^{2^{\ast }/2}}\left( \int_{\Omega _{r}}|\nabla
u|^{2}dx\right) ^{\frac{N}{N-2}}.  \label{e28}
\end{eqnarray}%
Set
\begin{equation}
\widetilde{h}(t):=\frac{1}{2}\left( 1-\Vert V_{-}\Vert
_{_{N/2}}S^{-1}\right) t^{2}-\frac{C_{q}a^{\frac{2q-N(q-2)}{4}}}{q}t^{\frac{%
N(q-2)}{2}}-\frac{1}{2^{\ast }S^{2^{\ast }/2}}t^{2^{\ast }}\text{ for }t>0.
\label{e29}
\end{equation}%
Then for any $0<a<\widetilde{a}_{V}$, there exist $0<\widetilde{t}%
_{1,a}<T_{a}<\widetilde{t}_{2,a}$ such that $\widetilde{h}(t)<0$ for $t\in
(0,\widetilde{t}_{1,a})\cup (\widetilde{t}_{2,a},\infty )$, $\widetilde{h}%
(t)>0$ for $t\in (\widetilde{t}_{1,a},\widetilde{t}_{2,a})$, and%
\begin{equation}
\widetilde{h}(T_{a})=\max_{t>0}\widetilde{h}(t)>0.  \label{e30}
\end{equation}%
Moreover, we have%
\begin{equation}
T_{a}\rightarrow S^{\frac{N}{4}}(1-\Vert V_{-}\Vert _{_{N/2}}S^{-1})^{\frac{%
N-2}{4}}\text{ as }a\rightarrow 0^{+}.  \label{e31}
\end{equation}%
Define
\begin{equation*}
\mathbf{B}_{r,a}:=\{u\in S_{r,a}:\Vert \nabla u\Vert _{2}\leq T_{a}\}.
\end{equation*}%
Then we have the following results.

\begin{theorem}
\label{L3.1} Let $\Omega $ be a bounded smooth domain. Assume that condition
$(V_{0})$ holds. Then for any $0<a<\widetilde{a}_{V},$ the following
statements are true.\newline
$(i)$ If $r<\frac{\sqrt{\theta a}}{T_{a}}$, then $\mathbf{B}_{r,a}=\emptyset
$;\newline
$(ii)$ If
\begin{equation*}
r>\widetilde{r}_{a}:=\max \left\{ \frac{\sqrt{\theta a}}{T_{a}},\left[ \frac{%
q\theta (1+\Vert V\Vert _{N/2}S^{-1})}{2}a^{\frac{2-q}{2}}|\Omega |^{\frac{%
q-2}{2}}\right] ^{\frac{2}{N(q-2)-4}}\right\} ,
\end{equation*}%
then $\mathbf{B}_{r,a}\neq \emptyset $ and
\begin{equation*}
m_{r}(a):=\inf\limits_{u\in \mathbf{B}_{r,a}}I_{r}(u)<0
\end{equation*}%
is achieved by some interior point $\widetilde{u}_{r}>0$ in $\mathbf{B}%
_{r,a} $. In particular, $\widetilde{u}_{r}$ is a solution of problem (\ref%
{e4}) for some Lagrange multiplier $\widetilde{\lambda }_{r}\in \mathbb{R}$.
Furthermore, it holds $\liminf\limits_{r\rightarrow \infty }\widetilde{%
\lambda }_{r}>0$.
\end{theorem}

\begin{proof}
$(i)$ Let $v_{1}\in S_{1,a}$ be the positive eigenfunction associated to $%
\theta $. Then by the Poincar\'{e} inequality one has
\begin{equation*}
\int_{\Omega _{r}}|\nabla u|^{2}dx\geq \frac{\theta a}{r^{2}}\text{ for any }%
u\in S_{r,a}.
\end{equation*}%
Using this, together with $T_{a}$ being independent of $r$, we obtain that $%
\mathbf{B}_{r,a}=\emptyset $ if and only if $r<\frac{\sqrt{\theta a}}{T_{a}}%
. $\newline
$(ii)$ Define%
\begin{equation}
v_{r}(x):=r^{-\frac{N}{2}}v_{1}(r^{-1}x)\text{ for }x\in \Omega _{r}.
\label{e32}
\end{equation}%
Clearly, $v_{r}\in S_{r,a}.$ Note that
\begin{equation*}
\int_{\Omega }|\nabla v_{1}|^{2}dx=\theta a\text{ and }\int_{\Omega
}|v_{1}|^{q}dx\geq a^{\frac{q}{2}}|\Omega |^{\frac{2-q}{2}}.
\end{equation*}%
Then a direct calculation shows that
\begin{equation}
\int_{\Omega _{r}}|\nabla v_{r}|^{2}dx=r^{-2}\int_{\Omega }|\nabla
v_{1}|^{2}dx=r^{-2}\theta a  \label{e33}
\end{equation}%
and
\begin{equation}
\int_{\Omega _{r}}|v_{r}|^{q}dx=r^{\frac{N(2-q)}{2}}\int_{\Omega
}|v_{1}|^{q}dx\geq r^{\frac{N(2-q)}{2}}a^{\frac{q}{2}}|\Omega |^{\frac{2-q}{2%
}}.  \label{e34}
\end{equation}%
It follows from condition $(V_{0})$ and (\ref{e33})--(\ref{e34}) that for
any $r>\widetilde{r}_{a}$,
\begin{eqnarray}
I_{r}(v_{r})& =&\frac{1}{2}\int_{\Omega _{r}}|\nabla v_{r}|^{2}dx+\frac{1}{2}%
\int_{\Omega _{r}}V(x)v_{r}^{2}dx-\frac{1}{2^{\ast }}\int_{\Omega
_{r}}|v_{r}|^{2^{\ast }}dx-\frac{1}{q}\int_{\Omega _{r}}|v_{r}|^{q}dx  \notag
\\
& \leq& \frac{1}{2}\left( 1+\Vert V\Vert _{N/2}S^{-1}\right) r^{-2}\theta a-%
\frac{1}{q}r^{\frac{N(2-q)}{2}}a^{\frac{q}{2}}|\Omega |^{\frac{2-q}{2}}<0.
\label{e35}
\end{eqnarray}%
Moreover, by (\ref{e28}) we obtain that the energy functional $I_{r}$ is
bounded from below in $\mathbf{B}_{r,a}$. So,%
\begin{equation*}
m_{r}(a)=\inf\limits_{u\in \mathbf{B}_{r,a}}I_{r}(u)<0.
\end{equation*}%
According to the Ekeland variational principle, there exists a sequence $%
\{u_{n}\}\subset \mathbf{B}_{r,a}$ such that
\begin{equation*}
I_{r}(u_{n})\rightarrow m_{r}(a)\ \text{and }I_{r}^{\prime
}(u_{n})|_{T_{u_{n}}S_{r,a}}\rightarrow 0\ \text{as}\ n\rightarrow \infty .
\end{equation*}%
Since $\Vert \nabla u_{n}\Vert _{2}\leq T_{a}$, there exists $\widetilde{u}%
_{r}\in H_{0}^{1}(\Omega _{r})$ such that
\begin{equation*}
\begin{array}{l}
u_{n}\rightharpoonup \widetilde{u}_{r}\ \text{in}\ H_{0}^{1}(\Omega _{r}),
\\
u_{n}\rightarrow \widetilde{u}_{r}\ \text{in}\ L^{k}(\Omega _{r})\ \text{for
all}\ 2\leq k<2^{\ast }.%
\end{array}%
\end{equation*}%
Moreover,
\begin{equation*}
\Vert \nabla \widetilde{u}_{r}\Vert _{2}^{2}\leq
\liminf\limits_{n\rightarrow \infty }\Vert \nabla u_{n}\Vert _{2}^{2}\leq
T_{a}^{2},
\end{equation*}%
that is, $\widetilde{u}_{r}\in \mathbf{B}_{r,a}$. Note that
\begin{equation*}
\int_{\Omega _{r}}V(x)u_{n}^{2}dx\rightarrow \int_{\Omega _{r}}V(x)%
\widetilde{u}_{r}^{2}dx\ \text{as}\ n\rightarrow \infty .
\end{equation*}%
Let $\widetilde{v}_{n}=u_{n}-\widetilde{u}_{r}$, then we have
\begin{equation*}
I_{r}(u_{n})=I_{r}(\widetilde{u}_{r})+\frac{1}{2}\int_{\Omega _{r}}|\nabla
\widetilde{v}_{n}|^{2}dx-\frac{1}{2^{\ast }}\int_{\Omega _{r}}|\widetilde{v}%
_{n}|^{2^{\ast }}dx+o_{n}(1).
\end{equation*}%
Since $I_{r}^{\prime }(u_{n})|_{T_{u_{n}}S_{r,a}}\rightarrow 0$ as $%
n\rightarrow \infty $, there exists $\{\widetilde{\lambda }_{n}\}$ such that
$I_{r}^{\prime }(u_{n})-\widetilde{\lambda }_{n}u_{n}\rightarrow 0$ as $%
n\rightarrow \infty $. Let $\widetilde{\lambda }_{r}$ be the lagrange
multiplier corresponding to $\widetilde{u}_{r}$. Then for some $\phi \in
H_{0}^{1}(\Omega _{r})$ with $\int_{\Omega _{r}}\widetilde{u}_{r}\phi dx\neq
0$, we have
\begin{equation*}
\widetilde{\lambda }_{n}=\frac{1}{\int_{\Omega _{r}}u_{n}\phi dx}\left(
\langle I_{r}^{\prime }(u_{n}),\phi \rangle +o_{n}(1)\right) \rightarrow
\frac{1}{\int_{\Omega _{r}}\widetilde{u}_{r}\phi dx}\langle I_{r}^{\prime }(%
\widetilde{u}_{r}),\phi \rangle =\widetilde{\lambda }_{r}.
\end{equation*}%
Using the Br\'{e}zis-Lieb Lemma and the facts that
\begin{equation*}
I_{r}^{\prime }(u_{n})-\widetilde{\lambda }_{n}u_{n}\rightarrow 0\ \text{as }%
n\rightarrow \infty \text{ and }I_{r}^{\prime }(\widetilde{u}_{r})-%
\widetilde{\lambda }_{r}\widetilde{u}_{r}=0,
\end{equation*}%
we deduce that%
\begin{equation*}
\int_{\Omega _{r}}|\nabla \widetilde{v}_{n}|^{2}dx=\int_{\Omega _{r}}|%
\widetilde{v}_{n}|^{2^{\ast }}dx+o_{n}(1).
\end{equation*}%
So we can assume that $\int_{\Omega _{r}}|\nabla \widetilde{v}%
_{n}|^{2}dx\rightarrow \ell \geq 0$ and $\int_{\Omega_r}|\widetilde{v}%
_{n}|^{2^{\ast }}dx\rightarrow \ell \geq 0$. Then we have
\begin{equation*}
\frac{1}{2}\int_{\Omega _{r}}|\nabla \widetilde{v}_{n}|^{2}dx-\frac{1}{%
2^{\ast }}\int_{\Omega _{r}}|\widetilde{v}_{n}|^{2^{\ast }}dx=\frac{\ell }{N}%
+o_{n}(1),
\end{equation*}%
which implies that%
\begin{equation*}
m_{r}(a)=\lim\limits_{n\rightarrow \infty }I_{r}(u_{n})\geq I_{r}(\widetilde{%
u}_{r}).
\end{equation*}%
This indicates that $I_{r}(\widetilde{u}_{r})=m_{r}(a)$, which in turn shows
that $\ell =0$ and thus $u_{n}\rightarrow \widetilde{u}_{r}$ in $%
H_{0}^{1}(\Omega _{r})$. Therefore, $\widetilde{u}_{r}$ is an interior point
of $\mathbf{B}_{r,a},$ since $I_{r}(u)>0$ for any $u\in \partial \mathbf{B}%
_{r,a}$ by (\ref{e28})--(\ref{e30}). The lagrange multiplier theorem implies
that $(\widetilde{\lambda }_{r},\widetilde{u}_{r})$ is a solution of problem
(\ref{e4}). Moreover,
\begin{eqnarray*}
\widetilde{\lambda }_{r}a& =&\int_{\Omega _{r}}|\widetilde{u}_{r}|^{2^{\ast
}}dx+\int_{\Omega _{r}}|\widetilde{u}_{r}|^{q}dx-\int_{\Omega _{r}}|\nabla
\widetilde{u}_{r}|^{2}dx-\int_{\Omega _{r}}V(x)\widetilde{u}_{r}^{2}dx \\
& =&\frac{(q-2)}{q}\int_{\Omega _{r}}|\widetilde{u}_{r}|^{q}dx+\frac{1}{N}%
\int_{\Omega _{r}}|\widetilde{u}_{r}|^{2^{\ast }}dx-2I_{r}(\widetilde{u}_{r})
\\
& >&-2I_{r}(\widetilde{u}_{r})=-2m_{r}(a).
\end{eqnarray*}%
Using the strong maximum principle, we obtain $\widetilde{u}_{r}>0$.
According to the definition of $m_{r}(a),$ we know that $m_{r}(a)$ is
non-increasing on $r$, which implies that $m_{r}(a)\leq m_{\widetilde{r}%
_{a}}(a)<0$ for any $0<a<\widetilde{a}_{V}$ and $r>\widetilde{r}_{a}$. This
shows that $\liminf\limits_{r\rightarrow \infty }\widetilde{\lambda }_{r}>0.$
The proof is complete.
\end{proof}

\textbf{We are ready to prove Theorem \ref{T1.4}: }$(i)$ It is a direct
consequence of Theorem \ref{L3.1}.

$(ii)$\textbf{\ }On the contrary, we assume that there exists $\widetilde{w}%
_{r}\in S_{r,a}$ such that%
\begin{equation*}
I_{r}^{\prime }(\widetilde{w}_{r})|_{S_{r,a}}=0\ \text{and }I_{r}(\widetilde{%
w}_{r})<m_{r}(a)=I_{r}(\widetilde{u}_{r}).
\end{equation*}%
Then $\widetilde{w}_{r}$ satisfies the equation%
\begin{equation*}
-\Delta u+V(x)u+\lambda u=|u|^{q-2}u+|u|^{2^{\ast }-2}u,\ x\in\Omega_r
\end{equation*}%
for some $\lambda \in
\mathbb{R}
.$ This indicates that $\widetilde{w}_{r}$ satisfies the Pohozaev identity%
\begin{equation*}
\int_{\Omega _{r}}|\nabla \widetilde{w}_{r}|^{2}dx-\frac{1}{2}\int_{\partial
\Omega _{r}}|\nabla \widetilde{w}_{r}|^{2}(x\cdot n)d\delta -\frac{1}{2}%
\int_{\Omega _{r}}(\nabla V\cdot x)\widetilde{w}_{r}^{2}dx-\int_{\Omega
_{r}}|\widetilde{w}_{r}|^{2^{\ast }}dx-\frac{N(q-2)}{2q}\int_{\Omega _{r}}|%
\widetilde{w}_{r}|^{q}dx=0.
\end{equation*}%
Since $\Omega $ is a star-shaped bounded domain, we have%
\begin{equation*}
\int_{\Omega _{r}}|\nabla \widetilde{w}_{r}|^{2}dx>\frac{1}{2}\int_{\Omega
_{r}}(\nabla V\cdot x)\widetilde{w}_{r}^{2}dx+\int_{\Omega _{r}}|\widetilde{w%
}_{r}|^{2^{\ast }}dx+\frac{N(q-2)}{2q}\int_{\Omega _{r}}|\widetilde{w}%
_{r}|^{q}dx.
\end{equation*}%
Using this inequality one has%
\begin{eqnarray*}
I_{r}(\widetilde{w}_{r}) &=&\frac{1}{2}\int_{\Omega _{r}}|\nabla \widetilde{w%
}_{r}|^{2}dx+\frac{1}{2}\int_{\Omega _{r}}V(x)\widetilde{w}_{r}^{2}dx-\frac{1%
}{2^{\ast }}\int_{\Omega _{r}}|\widetilde{w}_{r}|^{2^{\ast }}dx-\frac{1}{q}%
\int_{\Omega _{r}}|\widetilde{w}_{r}|^{q}dx \\
&>&\frac{1}{4}\int_{\Omega _{r}}(\nabla V\cdot x)\widetilde{w}_{r}^{2}dx+%
\frac{1}{2}\int_{\Omega _{r}}|\widetilde{w}_{r}|^{2^{\ast }}dx+\frac{N(q-2)}{%
4q}\int_{\Omega _{r}}|\widetilde{w}_{r}|^{q}dx \\
&&+\frac{1}{2}\int_{\Omega _{r}}V(x)\widetilde{w}_{r}^{2}dx-\frac{1}{2^{\ast
}}\int_{\Omega _{r}}|\widetilde{w}_{r}|^{2^{\ast }}dx-\frac{1}{q}%
\int_{\Omega _{r}}|\widetilde{w}_{r}|^{q}dx \\
&=&\frac{1}{4}\int_{\Omega _{r}}\left[ 2V(x)+(\nabla V\cdot x)\right]
\widetilde{w}_{r}^{2}dx+\frac{1}{N}\int_{\Omega _{r}}|\widetilde{w}%
_{r}|^{2^{\ast }}dx-\frac{4-N(q-2)}{4q}\int_{\Omega _{r}}|\widetilde{w}%
_{r}|^{q}dx,
\end{eqnarray*}%
which implies that%
\begin{equation}
\int_{\Omega _{r}}|\widetilde{w}_{r}|^{2^{\ast }}dx<NI_{r}(\widetilde{w}%
_{r})-\frac{N}{4}\int_{\Omega _{r}}\left[ 2V(x)+(\nabla V\cdot x)\right]
\widetilde{w}_{r}^{2}dx+\frac{N\left( 4-N(q-2)\right) }{4q}\int_{\Omega
_{r}}|\widetilde{w}_{r}|^{q}dx.  \label{e36}
\end{equation}%
It follows from (\ref{e36}) and the Gagliardo-Nirenberg inequality that%
\begin{eqnarray*}
&&\int_{\Omega _{r}}|\nabla \widetilde{w}_{r}|^{2}dx \\
&=&2I_{r}(\widetilde{w}_{r})-\int_{\Omega _{r}}V(x)\widetilde{w}_{r}^{2}dx+%
\frac{N-2}{N}\int_{\Omega _{r}}|\widetilde{w}_{r}|^{2^{\ast }}dx+\frac{2}{q}%
\int_{\Omega _{r}}|\widetilde{w}_{r}|^{q}dx \\
&<&2I_{r}(\widetilde{w}_{r})-\int_{\Omega _{r}}V(x)\widetilde{w}_{r}^{2}dx+%
\frac{2}{q}\int_{\Omega _{r}}|\widetilde{w}_{r}|^{q}dx \\
&&+\frac{N-2}{N}\left[ NI_{r}(\widetilde{w}_{r})-\frac{N}{4}\int_{\Omega
_{r}}\left[ 2V(x)+(\nabla V\cdot x)\right] \widetilde{w}_{r}^{2}dx+\frac{N%
\left( 4-N(q-2)\right) }{4q}\int_{\Omega _{r}}|\widetilde{w}_{r}|^{q}dx%
\right] \\
&=&NI_{r}(\widetilde{w}_{r})-\int_{\Omega _{r}}V(x)\widetilde{w}_{r}^{2}dx-%
\frac{N-2}{4}\int_{\Omega _{r}}\left[ 2V(x)+(\nabla V\cdot x)\right]
\widetilde{w}_{r}^{2}dx \\
&&+\left( \frac{2}{q}+\frac{N(N-2)\left( 4-N(q-2)\right) }{4Nq}\right)
\int_{\Omega _{r}}|\widetilde{w}_{r}|^{q}dx \\
&=&NI_{r}(\widetilde{w}_{r})-\frac{1}{4}\int_{\Omega _{r}}\left[
2NV(x)+(N-2)(\nabla V\cdot x)\right] \widetilde{w}_{r}^{2}dx+\frac{%
N(2N-q(N-2))}{4q}\int_{\Omega _{r}}|\widetilde{w}_{r}|^{q}dx \\
&\leq &NI_{r}(\widetilde{w}_{r})-\frac{1}{4}\int_{\Omega _{r}}\left[
2NV(x)+(N-2)(\nabla V\cdot x)\right] \widetilde{w}_{r}^{2}dx \\
&&+\frac{C_{q}a^{\frac{2q-N(q-2)}{4}}N(2N-q(N-2))}{4q}\left( \int_{\Omega
_{r}}|\nabla \widetilde{w}_{r}|^{2}dx\right) ^{\frac{N(q-2)}{4}},
\end{eqnarray*}%
which shows that%
\begin{eqnarray}
I_{r}(\widetilde{w}_{r}) &>&\frac{1}{N}\int_{\Omega _{r}}|\nabla \widetilde{w%
}_{r}|^{2}dx-\frac{C_{q}a^{\frac{2N-q(N-2)}{4}}(2N-q(N-2))}{4q}\left(
\int_{\Omega _{r}}|\nabla \widetilde{w}_{r}|^{2}dx\right) ^{\frac{N(q-2)}{4}}
\notag \\
&&-\frac{a}{2}\left( \Vert V\Vert _{\infty }+\frac{N-2}{2N}\Vert \widetilde{V%
}\Vert _{\infty }\right) ,  \label{e37}
\end{eqnarray}%
where we have used the fact that $V$ and $\widetilde{V}$ are both bounded.
Thus, by (\ref{e32}), (\ref{e35}) and (\ref{e37}) one has%
\begin{eqnarray*}
&&\frac{1}{N}\int_{\Omega _{r}}|\nabla \widetilde{w}_{r}|^{2}dx-\frac{%
C_{q}a^{\frac{2N-q(N-2))}{4}}(2N-q(N-2))}{4q}\left( \int_{\Omega
_{r}}|\nabla \widetilde{w}_{r}|^{2}dx\right) ^{\frac{N(q-2)}{4}} \\
&&-\frac{a}{2}\left( \Vert V\Vert _{\infty }+\frac{N-2}{2N}\Vert \widetilde{V%
}\Vert _{\infty }\right) \\
&\leq &I_{r}(\widetilde{w}_{r})<m_{r}(a)\leq I_{r}(v_{r}) \\
&=&\frac{1}{2}\int_{\Omega _{r}}|\nabla v_{r}|^{2}dx+\frac{1}{2}\int_{\Omega
_{r}}V(x)v_{r}^{2}dx-\frac{1}{2^{\ast }}\int_{\Omega _{r}}|v_{r}|^{2^{\ast
}}dx-\frac{1}{q}\int_{\Omega _{r}}|v_{r}|^{q}dx \\
&\leq &\frac{1}{2}\left( 1+\Vert V\Vert _{N/2}S^{-1}\right) r^{-2}\theta a-%
\frac{1}{q}r^{-\frac{N(q-2)}{2}}a^{\frac{q}{2}}|\Omega |^{\frac{2-q}{2}} \\
&<&0\text{ for any }r>\widetilde{r}_{a}.
\end{eqnarray*}%
This implies that%
\begin{equation}
\int_{\Omega _{r}}|\nabla \widetilde{w}_{r}|^{2}dx\rightarrow 0\text{ as }%
a\rightarrow 0^{+}.  \label{e38}
\end{equation}%
Hence, it follows from (\ref{e31}) and (\ref{e38}) that there exists $0<%
\widetilde{a}_{\ast }\leq \widetilde{a}_{V}$ such that for any $0<a<%
\widetilde{a}_{\ast },$
\begin{equation*}
\int_{\Omega _{r}}|\nabla \widetilde{w}_{r}|^{2}dx\leq T_{a}\text{ for }r>%
\widetilde{r}_{a}.
\end{equation*}%
This indicates that $\widetilde{w}_{r}\in \mathbf{B}_{r,a},$ which is a
contradiction. The proof is complete.

\subsection{The high-energy solution}

In order to find a high-energy solution, we need to introduce a modified
problem related to problem (\ref{e4}):
\begin{equation}
\left\{
\begin{array}{ll}
-\Delta u+V(x)u+\lambda u=s|u|^{q-2}u+s|u|^{2^{\ast }-2}u & \ \text{in}\
\Omega _{r}, \\
u\in H_{0}^{1}(\Omega _{r}),\ \int_{\Omega _{r}}|u|^{2}dx=a, &
\end{array}%
\right.   \label{e39}
\end{equation}%
where $\frac{1}{2}\leq s\leq 1$. Clearly, solutions of problem (\ref{e39})
correspond to critical points of the energy functional $\widetilde{I}%
_{r,s}:S_{r,a}\rightarrow \mathbb{R}$ defined by
\begin{equation}
\widetilde{I}_{r,s}(u)=\frac{1}{2}\int_{\Omega _{r}}|\nabla u|^{2}dx+\frac{1%
}{2}\int_{\Omega _{r}}V(x)u^{2}dx-\frac{s}{q}\int_{\Omega _{r}}|u|^{q}dx-%
\frac{s}{2^{\ast }}\int_{\Omega _{r}}u^{2^{\ast }}dx.  \label{e40}
\end{equation}%
on the constraint $S_{r,a}.$ Note that $\widetilde{I}_{r,s}(u)=I_{r}(u)$
with $s=1.$ Similar to Theorem \ref{T1.5} $(i),$ problem (\ref{e39}) admits
a ground state solution $\widetilde{u}_{r,s}$ satisfying%
\begin{equation}
\widetilde{m}_{r,s}(a):=\widetilde{I}_{r,s}(\widetilde{u}_{r,s})<0
\label{e41}
\end{equation}%
for some $\widetilde{\lambda }_{r,s}\in\mathbb{R}.$ Now, we study the mountain pass geometry of the energy functional $%
\widetilde{I}_{r,s}$.

\begin{lemma}
\label{L3.2} Assume that condition $(V_{0})$ holds. Then for each $0<a<%
\widetilde{a}_{V},$ there exist $\widehat{r}_{a}>0$ and $\widetilde{u}_{0,s},%
\widetilde{u}_{1,s}\in S_{r,a}$ such that for $r>\widehat{r}_{a}$ and $s\in %
\left[ \frac{1}{2},1\right] ,$ the following statements are true.\newline
$(i)$ $\widetilde{I}_{r,s}(\widetilde{u}_{1,s})<0$ and $\widetilde{I}_{r,s}(%
\widetilde{u}_{0,s})<\widetilde{h}(T_{a})$. Moreover, $\Vert \nabla
\widetilde{u}_{0,s}\Vert _{2}<T_{a}<\Vert \nabla \widetilde{u}_{1,s}\Vert
_{2},$ where $\widetilde{h}(T_{a})$ is given in (\ref{e30});\newline
$(ii)$ there holds
\begin{equation*}
0<\widetilde{h}(T_{a})\leq \widetilde{M}_{r,s}(a):=\inf_{\gamma \in
\widetilde{\Gamma }_{r,a}}\sup_{t\in \lbrack 0,1]}\widetilde{I}_{r,s}(\gamma
(t))\leq \varphi (\widetilde{t}_{a}),
\end{equation*}%
where $\widetilde{\Gamma }_{r,a}:=\{\gamma \in C([0,1],S_{r,a}):\gamma (0)=%
\widetilde{u}_{0,s},\ \gamma (1)=\widetilde{u}_{1,s}\},$ and $\varphi (%
\widetilde{t}_{a})=\max\limits_{t>0}\varphi (t)$ with the function $\varphi :%
\mathbb{R}^{+}\rightarrow \mathbb{R}$ being defined by
\begin{equation*}
\varphi (t):=\frac{1}{2}\left( 1+\Vert V\Vert _{N/2}S^{-1}\right) \theta
at^{2}-\frac{1}{22^{\ast }}a^{2^{\ast }/2}|\Omega |^{-\frac{2}{N-2}%
}t^{2^{\ast }}
\end{equation*}%
and
\begin{equation*}
\widetilde{t}_{a}:=2\left[ \left( 1+\Vert V\Vert _{_{N/2}}S^{-1}\right)
\theta |\Omega |^{\frac{2}{N-2}}\right] ^{\frac{N-2}{4}}a^{-\frac{1}{2}}.
\end{equation*}
\end{lemma}

\begin{proof}
$(i)$ Let $v_{1}\in S_{1,a}$ be the positive eigenfunction associated with $%
\theta $. By (\ref{e6}), (\ref{e7}) and the Gagliardo-Nirenberg inequality
one has%
\begin{eqnarray*}
\widetilde{I}_{1/t,s}(v_{t})& =&\frac{1}{2}\int_{\Omega _{1/t}}|\nabla
v_{t}|^{2}dx+\frac{1}{2}\int_{\Omega _{1/t}}V(x)v_{t}^{2}dx-\frac{s}{2^{\ast
}}\int_{\Omega _{1/t}}|v_{t}|^{2^{\ast }}dx-\frac{s}{p}\int_{\Omega
_{1/t}}|v_{t}|^{p}dx \\
& =&\frac{t^{2}}{2}\int_{\Omega }|\nabla v_{1}|^{2}dx+\frac{1}{2}\int_{\Omega
}V(x/t)|v_{1}|^{2}dx-\frac{st^{2^{\ast }}}{2^{\ast }}\int_{\Omega
}|v_{1}|^{2^{\ast }}dx-\frac{st^{\frac{N(q-2)}{2}}}{q}\int_{\Omega
}|v_{1}|^{q}dx \\
& \leq& \frac{1}{2}\left( 1+\Vert V\Vert _{N/2}S^{-1}\right) \theta at^{2}-%
\frac{1}{22^{\ast }}a^{2^{\ast }/2}|\Omega |^{-\frac{2}{N-2}}t^{2^{\ast
}}\\
&=:&\varphi (t).
\end{eqnarray*}%
By calculating, there holds $\varphi (\widetilde{t}_{0})=0$, $\varphi
(t)<0$ for any $t>\widetilde{t}_{0}$ and $\varphi (t)>0$ for any $0<t<%
\widetilde{t}_{0}$, where%
\begin{equation*}
\widetilde{t}_{0}:=2^{\ast }\left[ \left( 1+\Vert V\Vert _{N/2}S^{-1}\right)
\theta |\Omega |^{\frac{2}{N-2}}\right] ^{\frac{N-2}{4}}a^{-\frac{1}{2}}.
\end{equation*}%
In addition, $\varphi $ achieves its maximum at $\widetilde{t}_{a},$ that
is, $\varphi (\widetilde{t}_{a})=\max\limits_{t>0}\varphi (t)$.

On the other hand, for any $\frac{1}{2}\leq s\leq 1,$ it follows from the
Gagliardo-Nirenberg and H\"{o}lder inequalities that%
\begin{eqnarray*}
\widetilde{I}_{r,s}(u) &\geq &I_{r}(u) \\
&\geq &\frac{1}{2}\left( 1-\Vert V_{-}\Vert _{N/2}S^{-1}\right) \int_{\Omega
_{r}}|\nabla u|^{2}dx-\frac{C_{q}a^{\frac{2q-N(q-2)}{4}}}{q}\left(
\int_{\Omega _{r}}|\nabla u|^{2}dx\right) ^{\frac{N(q-2)}{4}} \\
&&-\frac{1}{2^{\ast }S^{2^{\ast }/2}}\left( \int_{\Omega _{r}}|\nabla
u|^{2}dx\right) ^{\frac{N}{N-2}} \\
&=&\widetilde{h}\left( \int_{\Omega _{r}}|\nabla u|^{2}dx\right) ,
\end{eqnarray*}%
where $\widetilde{h}(t)$ is given as (\ref{e29}). Set
\begin{equation*}
\widetilde{t}_{2,a}:=\left[ \frac{C_{q}N(q-2)(4-N(q-2))(N-2)S^{2^{\ast }/2}}{%
16q}\right] ^{\frac{4(N-2)}{2N-q(N-2)}}a^{\frac{N-2}{N}}.
\end{equation*}%
By calculating the second derivative of $\widetilde{h},$ we find that $%
\widetilde{h}^{\prime \prime }(t)\leq 0$ if and only if $t\geq \widetilde{t}%
_{2,a},$ and so $\widetilde{t}_{2,a}<T_{a}$. Moreover, we note that for any $%
t\geq \widetilde{t}_{2,a}$,
\begin{eqnarray*}
\widetilde{h}(t)& \geq& \frac{1}{2}\left( 1-\Vert V_{-}\Vert
_{N/2}S^{-1}\right) t-\left[ \frac{16}{S^{2^{\ast }/2}(N-2)(q-2)(4-N(q-2))}+%
\frac{1}{2^{\ast }S^{2^{\ast }/2}}\right] t^{\frac{N}{N-2}} \\
& =:&\widetilde{g}(t).
\end{eqnarray*}%
Then, according to the definition of $\widetilde{a}_{V}$, there exists a
constant $\widetilde{t}_{3,a}>\widetilde{t}_{2,a}$ such that
\begin{eqnarray*}
\max_{t>0}\widetilde{h}(t)&\geq& \max_{t\geq \widetilde{t}_{2,a}}\widetilde{g}(t)=\widetilde{g}(\widetilde{t}_{3,a}) \\
& =&\frac{2}{N-2}\left[ \frac{(N-2)(1-\Vert V_{-}\Vert _{N/2}S^{-1})}{2NS}%
\right] ^{\frac{N}{2}}\left[ \frac{32+(N-2)^{2}(q-2)(4-N(q-2))}{2N(N-2)(q-2)(4-N(q-2))}\right] ^{\frac{2-N}{2}}.
\end{eqnarray*}%
According to the properties of $\varphi $, there exists $0<\widetilde{t}%
_{4,a}<\widetilde{t}_{a}$ such that for any $t\in (0,\widetilde{t}_{4,a}]$,
\begin{equation*}
\varphi (t)<\varphi (\widetilde{t}_{4,a})<\widetilde{h}(T_{a}).
\end{equation*}%
Thus for any%
\begin{equation*}
r>\widetilde{r}_{a}^{\ast }:=\max \left\{ \frac{1}{\widetilde{t}_{4,a}},%
\sqrt{\frac{\theta a}{\widetilde{t}_{3,a}}}\right\} ,
\end{equation*}%
we have $v_{1/\widetilde{r}_{a}^{\ast }}\in S_{r,a}$ and $\Vert \nabla v_{1/%
\widetilde{r}_{a}^{\ast }}\Vert _{2}^{2}=(\widetilde{r}_{a}^{\ast
})^{-2}\Vert \nabla v_{1}\Vert _{2}^{2}<T_{a},$ and%
\begin{equation*}
\widetilde{I}_{\widetilde{r}_{a}^{\ast },s}(v_{_{1/\widetilde{r}_{a}^{\ast
}}})\leq \widetilde{h}(v_{1/\widetilde{r}_{a}^{\ast }})\leq \widetilde{h}(%
\widetilde{t}_{4,a}).
\end{equation*}%
Set $\widetilde{u}_{0,s}:=v_{1/\widetilde{r}_{a}^{\ast }}$ and $\widetilde{u}%
_{1,s}:=v_{\widetilde{t}_{0}}$. Obviously, it holds $\Vert \nabla \widetilde{%
u}_{0,s}\Vert _{2}<T_{a}<\Vert \nabla \widetilde{u}_{1,s}\Vert _{2}$. Set $%
\widehat{r}_{a}:=\max \left\{ \frac{1}{\widetilde{t}_{0}},\widetilde{r}%
_{a}^{\ast },\widetilde{r}_{a}\right\} ,$ where $\widetilde{r}_{a}$ is given
in Lemma \ref{L3.1}, we complete the proof of $(i).$\newline
$(ii)$ Note that for any $\gamma \in \widetilde{\Gamma }_{r,a}$, there holds
\begin{equation*}
\Vert \nabla \gamma (0)\Vert _{2}<T_{a}<\Vert \nabla \gamma (1)\Vert _{2}.
\end{equation*}%
Then we have
\begin{equation*}
\sup_{t\in \lbrack 0,1]}\widetilde{I}_{r,s}(\gamma (t))\geq \max_{t>0}%
\widetilde{h}(t)=\widetilde{h}(T_{a}),
\end{equation*}%
which shows that $\widetilde{M}_{r,s}(a)=\inf_{\gamma \in \widetilde{\Gamma }%
_{r,a}}\sup_{t\in \lbrack 0,1]}\widetilde{I}_{r,s}(\gamma (t))\geq
\widetilde{h}(T_{a})$. Define a path $\widetilde{\gamma }:[0,1]\rightarrow
S_{r,a}$ by
\begin{equation*}
\widetilde{\gamma }(T):\Omega _{r}\rightarrow \mathbb{R},\ x\mapsto \left[ T%
\widetilde{t}_{0}+\frac{(1-T)}{\widehat{r}_{a}}\right] ^{\frac{N}{2}%
}v_{1}\left( \left( T\widetilde{t}_{0}+\frac{(1-T)}{\widehat{r}_{a}}\right)
x\right) .
\end{equation*}%
Clearly, $\widetilde{\gamma }(T)\in \widetilde{\Gamma }_{r,a}$ and $%
\sup_{T\in \lbrack 0,1]}\widetilde{I}_{r,s}(\widetilde{\gamma }(T))\leq
\varphi (\widetilde{t}_{a})$. The proof is complete.
\end{proof}

Next, we give the estimates of $\widetilde{M}_{r,s}(a)$. Without loss of
generality, we assume $0\in \Omega _{r}$. For every $\varepsilon >0$, let us
introduce the function $U_{\varepsilon }\in H_{0}^{1}(\Omega _{r})$ defined
as
\begin{equation*}
U_{\varepsilon }(x)=\frac{\eta \lbrack N(N-2)\varepsilon ^{2}]^{\frac{N-2}{4}%
}}{[\varepsilon ^{2}+|x|^{2}]^{\frac{N-2}{2}}},
\end{equation*}%
where $\eta \in C_{0}^{\infty }(\Omega _{r})$ is a cut-off function such
that $\eta (x)\equiv 1$ for $0\leq |x|\leq R$, $0\leq \eta (x)\leq 1$ for $%
R\leq |x|\leq 2R$, $\eta (x)\equiv 0$ for $|x|\geq 2R$, where $B_{2R}\subset
\Omega _{r}$. Using the estimates by \cite{S2,DHPZ}, we have
\begin{equation*}
\begin{array}{l}
\Vert \nabla U_{\varepsilon }\Vert _{2,\Omega
_{r}}^{2}=S^{N/2}+O(\varepsilon ^{N-2}); \\
\Vert U_{\varepsilon }\Vert _{2^{\ast },\Omega _{r}}^{2^{\ast
}}=S^{N/2}+O(\varepsilon ^{N}); \\
\Vert U_{\varepsilon }\Vert _{2,\Omega _{r}}^{2}=\left\{
\begin{array}{ll}
C_{1}\varepsilon ^{2}+O(\varepsilon ^{N-2}) & \ \text{if}\ N\geq 5, \\
C_{2}\varepsilon ^{2}|\ln \varepsilon |+O(\varepsilon ^{2}) & \ \text{if}\
N=4, \\
C_{3}\varepsilon +O(\varepsilon ^{2}) & \ \text{if}\ N=3.%
\end{array}%
\right.
\end{array}%
\end{equation*}%
Moreover, when $\varepsilon \rightarrow 0^{+},$ we have
\begin{equation*}
\Vert U_{\varepsilon }\Vert _{q,\Omega _{r}}^{q}\sim \left\{
\begin{array}{ll}
\varepsilon ^{(6-q)/2} & \ \text{for }3<q<6, \\
\varepsilon ^{3/2}|\ln \varepsilon | & \text{ for }q=3, \\
\varepsilon ^{q/2} & \ \text{for }2<q<3,%
\end{array}%
\right.
\end{equation*}%
if $N=3,$ and
\begin{equation*}
\Vert U_{\varepsilon }\Vert _{q,\Omega _{r}}^{q}\sim \varepsilon ^{\frac{%
2N-q(N-2)}{2}}\ \text{for }2<q<2^{\ast },
\end{equation*}%
if $N\geq 4.$

\begin{lemma}
\label{L3.3} (\cite[Lemma 5.5]{SZ}) For any $\phi \in L_{loc}^{\infty
}(\Omega _{r})$, we have as $\varepsilon \rightarrow 0^{+}$,
\begin{equation*}
\int_{\Omega _{r}}\phi U_{\varepsilon }dx \leq 2[N(N-2)]^{\frac{N-2}{2}%
}\sup\limits_{B_{2R}}\phi w_{N}R^{2}\varepsilon ^{\frac{N-2}{2}}+o\left(
\varepsilon ^{\frac{N-2}{2}}\right)
\end{equation*}
and
\begin{equation*}
\int_{\Omega _{r}}\phi U_{\varepsilon }^{2^{\ast }-1}dx \geq \frac{1}{4}%
[N(N-2)]^{\frac{N+2}{2}}\sup\limits_{B_{2R}}\phi w_{N}\varepsilon ^{\frac{N-2%
}{2}}+o\left( \varepsilon ^{\frac{N-2}{2}}\right) .
\end{equation*}
\end{lemma}

\begin{lemma}
\label{L3.4}Assume that condition $(V_{0})$ holds and $\widetilde{V}(x)$ is
bounded on $\mathbb{R}^{N}$. Then for any $0<a<\widetilde{a}_{V}$ and $s\in %
\left[ \frac{1}{2},1\right] $, there holds
\begin{equation*}
\widetilde{M}_{r,s}(a)<\widetilde{m}_{r,s}(a)+\frac{1}{N}S^{\frac{N}{2}}s^{1-%
\frac{N}{2}},
\end{equation*}%
where $\widetilde{m}_{r,s}(a)$ is defined as (\ref{e41}).
\end{lemma}

\begin{proof}
Let%
\begin{equation*}
w_{\varepsilon ,t}:=\widetilde{u}_{r,s}+tU_{\varepsilon }\ \text{for }t\geq
0,
\end{equation*}%
where $\widetilde{u}_{r,s}$ is the ground state solution of problem (\ref%
{e39}) satisfying $\widetilde{m}_{r,s}(a)=\widetilde{I}_{r,s}(\widetilde{u}%
_{r,s})$. Clearly, $w_{\varepsilon ,t}>0$ in $\Omega _{r}.$ Define $%
W_{\varepsilon ,t}:=\mu ^{\frac{N-2}{2}}w_{\varepsilon ,t}(\mu x),$ where $%
\mu =\frac{\Vert w_{\varepsilon ,t}\Vert _{2,\Omega _{r}}}{\sqrt{a}}>0$.
Then there hold $W_{\varepsilon ,t}\in H_{0}^{1}(\Omega _{r})$ and $\Vert
W_{\varepsilon ,t}\Vert _{2,\Omega _{r}}^{2}=a$. A direct calculation shows
that
\begin{equation*}
\Vert \nabla W_{\varepsilon ,t}\Vert _{2}^{2}=\Vert \nabla w_{\varepsilon
,t}\Vert _{2}^{2},\ \Vert W_{\varepsilon ,t}\Vert _{2^{\ast }}^{2^{\ast
}}=\Vert w_{\varepsilon ,t}\Vert _{2^{\ast }}^{2^{\ast }}
\end{equation*}%
and
\begin{equation*}
\Vert W_{\varepsilon ,t}\Vert _{2}^{2}=\mu ^{-2}\Vert w_{\varepsilon
,t}\Vert _{2}^{2},\ \Vert W_{\varepsilon ,t}\Vert _{q}^{q}=\mu ^{q(\gamma
_{q}-1)}\Vert w_{\varepsilon ,t}\Vert _{q}^{q}.
\end{equation*}%
Here $\gamma _{q}=\frac{N(q-2)}{2q}<1$. Let
\begin{equation*}
(W_{\varepsilon ,\widetilde{t}})_{k}:=k^{\frac{N}{2}}W_{\varepsilon ,%
\widetilde{t}}(kx)\text{ for }k\geq 1,
\end{equation*}%
where $\widetilde{t}$ will be determined below. Define
\begin{eqnarray*}
\Phi (k):&=&\widetilde{I}_{r,s}((W_{\varepsilon ,\widetilde{t}})_{k})\\
&=&\frac{k^{2}}{2}\int_{\Omega _{r}}|\nabla W_{\varepsilon ,\widetilde{t}}|^{2}dx+%
\frac{1}{2}\int_{\Omega _{r}}V(x/k)|W_{\varepsilon ,\widetilde{t}}|^{2}dx \\
&&-\frac{sk^{2^{\ast }}}{2^{\ast }}\int_{\Omega _{r}}|W_{\varepsilon ,%
\widetilde{t}}|^{2^{\ast }}dx-\frac{sk^{\frac{N(q-2)}{2}}}{q}\int_{\Omega
_{r}}|W_{\varepsilon ,\widetilde{t}}|^{q}dx.
\end{eqnarray*}%
By calculating, we deduce that%
\begin{eqnarray*}
\Phi (k) &=&\frac{k^{2}}{2}\int_{\Omega _{r}}|\nabla \widetilde{u}%
_{r,s}|^{2}dx+\frac{1}{2}\int_{\Omega _{r}}V(x/k)|\widetilde{u}_{r,s}|^{2}dx-%
\frac{sk^{2^{\ast }}}{2^{\ast }}\int_{\Omega _{r}}|\widetilde{u}%
_{r,s}|^{2^{\ast }}dx \\
&&-\frac{sk^{\frac{N(q-2)}{2}}}{q}\int_{\Omega _{r}}|\widetilde{u}%
_{r,s}|^{q}dx+\frac{k^{2}\widetilde{t}^{2}}{2}S^{\frac{N}{2}}-\frac{%
sk^{2^{\ast }}\widetilde{t}^{2^{\ast }}}{2^{\ast }}S^{\frac{N}{2}%
}+o_{\varepsilon }(1)
\end{eqnarray*}%
and
\begin{eqnarray*}
\Phi ^{\prime }(k) &=&k\int_{\Omega _{r}}|\nabla \widetilde{u}_{r,s}|^{2}dx-%
\frac{1}{2k^{2}}\int_{\Omega _{r}}\nabla V(x/k)\cdot x|\widetilde{u}%
_{r,s}|^{2}dx-sk^{2^{\ast }-1}\int_{\Omega _{r}}|\widetilde{u}%
_{r,s}|^{2^{\ast }}dx \\
&&-\frac{N(q-2)sk^{\frac{N(q-2)-2}{2}}}{2q}\int_{\Omega _{r}}|\widetilde{u}%
_{r,s}|^{q}dx+k\widetilde{t}^{2}S^{\frac{N}{2}}-sk^{2^{\ast }-1}\widetilde{t}%
^{2^{\ast }}S^{\frac{N}{2}}+o_{\varepsilon }(1).
\end{eqnarray*}%
Since $\widetilde{V}(x)$ is bounded on $\mathbb{R}^{N}$, we can choose $%
\widetilde{t}$ large such that $\Phi ^{\prime }(k)<0$ for $k\geq 1$. This
shows that $\widetilde{I}_{r,s}((W_{\varepsilon ,\widetilde{t}})_{k})\leq
\widetilde{I}_{r,s}(W_{\varepsilon ,\widetilde{t}})$ for $k\geq 1$. So,
according to the definitition of $\widetilde{\Gamma }_{r,a}$, we have
\begin{equation}
\widetilde{M}_{r,s}(a)\leq \sup_{t\geq 0}\widetilde{I}_{r,s}(W_{\varepsilon
,t}).  \label{e42}
\end{equation}
A direct calculation shows that%
\begin{eqnarray}
\widetilde{I}_{r,s}(W_{\varepsilon ,t}) &=&\frac{1}{2}\Vert \nabla
w_{\varepsilon ,t}\Vert _{2}^{2}+\frac{\mu ^{-2}}{2}\Vert V(x/{\mu }%
)w_{\varepsilon ,t}\Vert _{2}^{2}-\frac{s}{2^{\ast }}\Vert w_{\varepsilon
,t}\Vert _{2^{\ast }}^{2^{\ast }}-\frac{s\mu ^{q(\gamma _{q}-1)}}{q}\Vert
w_{\varepsilon ,t}\Vert _{q}^{q}  \notag \\
&=&\widetilde{I}_{r,s}(w_{\varepsilon ,t})+\frac{s(1-\mu ^{q(\gamma _{q}-1)})%
}{q}\int_{\Omega _{r}}|w_{\varepsilon ,t}|^{q}dx+\frac{\mu ^{-2}}{2}%
\int_{\Omega _{r}}\left( V(x/{\mu })-\mu ^{2}V(x)\right) |w_{\varepsilon
,t}|^{2}dx  \notag \\
&\leq &\widetilde{I}_{r,s}(\widetilde{u}_{r,s})+\frac{t^{2}}{2}\int_{\Omega
_{r}}|\nabla U_{\varepsilon }|^{2}dx+\frac{t^{2}}{2}\int_{\Omega
_{r}}V(x)U_{\varepsilon }^{2}dx  \notag \\
&&+ts\int_{\Omega _{r}}|\widetilde{u}_{r,s}|^{2^{\ast }-2}\widetilde{u}%
_{r,s}U_{\varepsilon }dx+ts\int_{\Omega _{r}}|\widetilde{u}_{r,s}|^{q-2}%
\widetilde{u}_{r,s}U_{\varepsilon }dx  \notag \\
&&+\frac{s}{2^{\ast }}\int_{\Omega _{r}}\left( |\widetilde{u}%
_{r,s}|^{2^{\ast }}-|\widetilde{u}_{r,s}+tU_{\varepsilon }|^{2^{\ast
}}\right) dx+\frac{s}{q}\int_{\Omega _{r}}\left( |\widetilde{u}_{r,s}|^{q}-|%
\widetilde{u}_{r,s}+tU_{\varepsilon }|^{q}\right) dx  \notag \\
&&+\frac{s(1-\mu ^{q(\gamma _{q}-1)})}{q}\int_{\Omega _{r}}|w_{\varepsilon
,t}|^{q}dx+\frac{\mu ^{-2}}{2}\int_{\Omega _{r}}\left( V(x/{\mu })-\mu
^{2}V(x)\right) |w_{\varepsilon ,t}|^{2}dx,  \label{e43}
\end{eqnarray}%
since $\widetilde{u}_{r,s}$ is a ground state solution of problem (\ref{e39}%
). Note that%
\begin{equation*}
(1+b)^{\alpha }\geq 1+\alpha b^{\alpha -1}+b^{\alpha }\ \text{for any }b>0,
\end{equation*}%
where $\alpha \geq 2.$ Then for any $q>2,$ there holds%
\begin{equation}
\frac{1}{q}\int_{\Omega _{r}}\left( |\widetilde{u}_{r,s}|^{q}-|\widetilde{u}%
_{r,s}+tU_{\varepsilon }|^{q}\right) dx\leq -\frac{1}{q}\int_{\Omega
_{r}}|tU_{\varepsilon }|^{q}dx-t^{p-1}\int_{\Omega _{r}}\widetilde{u}%
_{r,s}U_{\varepsilon }^{q-1}dx.  \label{e44}
\end{equation}%
Since $\mu =\frac{\Vert w_{\varepsilon ,t}\Vert _{2,\Omega _{r}}}{\sqrt{a}},$
according to the definition of $w_{\varepsilon ,t}$, we have%
\begin{equation}
\mu ^{2}=1+\frac{2t}{a}\int_{\Omega _{r}}\widetilde{u}_{r,s}U_{\varepsilon
}dx+\frac{t^{2}}{a}\int_{\Omega _{r}}|U_{\varepsilon }|^{2}dx,  \label{e45}
\end{equation}%
which shows that $\mu \geq 1$ and $\lim_{\varepsilon \rightarrow 0}\mu =1$.
Since $\frac{4}{N}<q(1-\gamma _{q})<2$, it is clear that
\begin{equation*}
\lim\limits_{\mu \rightarrow 1^{+}}\frac{1-\mu ^{q(\gamma _{q}-1)}}{\mu
^{2}-1}=\frac{q(1-\gamma _{q})}{2},
\end{equation*}%
which implies that there exists $C>0$ such that
\begin{equation}
1-\mu ^{q(\gamma _{q}-1)}\leq C(\mu ^{2}-1)  \label{e46}
\end{equation}%
for $\mu \in (1,2)$. Let%
\begin{equation*}
\widetilde{f}(\mu ):=V(x/\mu )-\mu ^{2}V(x)\text{ for all }\mu \geq 1\text{
and }x\in\mathbb{R}^{N}.
\end{equation*}%
Clearly, $\widetilde{f}(1)=0$. A direct calculation gives
\begin{eqnarray*}
\lim_{\mu \rightarrow 1^{+}}\frac{\widetilde{f}(\mu )}{1-\mu ^{q(\gamma
_{q}-1)}}& =&\lim_{\mu \rightarrow 1^{+}}\frac{\nabla V(x/\mu )\cdot x+2\mu
^{3}V(x)}{q(\gamma _{q}-1)\mu ^{q(\gamma _{q}-1)+1}} \\
& =&\frac{\nabla V(x)\cdot x+2V(x)}{q(\gamma _{q}-1)}\text{ uniformly on }x\in\mathbb{R}^{N}.
\end{eqnarray*}%
This indicates that there exists $C>0$ such that
\begin{equation}
|\widetilde{f}(\mu )|\leq \left( C+\frac{|\nabla V(x)\cdot x+2V(x)|}{%
q(1-\gamma _{q})}\right) \left( 1-\mu ^{q(\gamma _{q}-1)}\right)  \label{e47}
\end{equation}%
for $\mu \in (1,2)$. Thus it follows from (\ref{e46}) and (\ref{e47}) that
for $\mu \in (1,2),$
\begin{eqnarray}
&&\frac{\mu ^{-2}}{2}\int_{\Omega _{r}}\left( V(x/{\mu })-\mu
^{2}V(x)\right) |w_{\varepsilon ,t}|^{2}dx  \notag \\
&=&\frac{\mu ^{-2}}{2}\int_{\Omega _{r}}\left( V(x/{\mu })-\mu
^{2}V(x)\right) \left( \widetilde{u}_{r,s}^{2}+2t\widetilde{u}%
_{r,s}U_{\varepsilon }+t^{2}U_{\varepsilon }^{2}\right) dx  \notag \\
&\leq &C(\mu ^{2}-1)\int_{\Omega _{r}}\left( C+\frac{|\nabla V(x)\cdot
x+2V(x)|}{q(1-\gamma _{q})}\right) \left( \widetilde{u}_{r,s}^{2}+2t%
\widetilde{u}_{r,s}U_{\varepsilon }+t^{2}U_{\varepsilon }^{2}\right) dx
\notag \\
&\leq &C(\mu ^{2}-1)+C(\mu ^{2}-1)\int_{\Omega _{r}}\left( C+\frac{|\nabla
V(x)\cdot x+2V(x)|}{q(1-\gamma _{q})}\right) \left( 2t\widetilde{u}%
_{r,s}U_{\varepsilon }+t^{2}U_{\varepsilon }^{2}\right) dx  \notag \\
&\leq &C(\mu ^{2}-1)+o(\varepsilon ^{\frac{N-2}{2}}).  \label{e48}
\end{eqnarray}%
Similarly, we also have%
\begin{equation}
\frac{s(1-\mu ^{q(\gamma _{q}-1)})}{q}\int_{\Omega _{r}}|w_{\varepsilon
,t}|^{q}dx\leq C(\mu ^{2}-1)+o(\varepsilon ^{\frac{N-2}{2}}).  \label{e49}
\end{equation}%
Hence, by (\ref{e43})-(\ref{e45}), (\ref{e48}), (\ref{e49}) and Lemma \ref%
{L3.3} one has%
\begin{align}\label{e50}
\widetilde{I}_{r,s}(W_{\varepsilon ,t}) \leq &\widetilde{I}_{r,s}(\widetilde{%
u}_{r,s})+\frac{t^{2}}{2}\int_{\Omega _{r}}|\nabla U_{\varepsilon }|^{2}dx+%
\frac{t^{2}}{2}\int_{\Omega _{r}}V(x)U_{\varepsilon }^{2}dx-\frac{st^{q}}{q}%
\int_{\Omega _{r}}|U_{\varepsilon }|^{q}dx  \notag \\
&-\frac{st^{2^{\ast }}}{2^{\ast }}\int_{\Omega _{r}}|U_{\varepsilon
}|^{2^{\ast }}dx+st\int_{\Omega _{r}}|\widetilde{u}_{r,s}|^{2^{\ast }-2}%
\widetilde{u}_{r,s}U_{\varepsilon }dx+st\int_{\Omega _{r}}|\widetilde{u}%
_{r,s}|^{q-2}\widetilde{u}_{r,s}U_{\varepsilon }dx  \notag \\
&+\frac{s(1-\mu ^{q(\gamma _{q}-1)})}{q}\int_{\Omega _{r}}|w_{\varepsilon
,t}|^{q}dx-st^{2^{\ast }-1}\int_{\Omega _{r}}\widetilde{u}%
_{r,s}U_{\varepsilon }^{2^{\ast }-1}dx  \notag \\
&+\frac{\mu ^{-2}}{2}\int_{\Omega _{r}}\left( V(x/{\mu })-\mu
^{2}V(x)\right) |w_{\varepsilon ,t}|^{2}dx  \notag \\
\leq &\widetilde{m}_{r,s}(a)+\frac{1}{N}S^{\frac{N}{2}}s^{1-\frac{N}{2}%
}+C\Vert v_{\varepsilon }\Vert _{2,\Omega _{r}}^{2}-C\Vert v_{\varepsilon
}\Vert _{q,\Omega _{r}}^{q}+(C_{1}R^{2}-C_{2})\varepsilon ^{\frac{N-2}{2}%
}+o(\varepsilon ^{\frac{N-2}{2}}).
\end{align}%
If $N=3$, then it follows from (\ref{e50}) that
\begin{equation*}
\widetilde{I}_{r,s}(W_{\varepsilon ,t})\leq \widetilde{m}_{r,s}(a)+\frac{1}{3%
}S^{\frac{3}{2}}s^{-\frac{1}{2}}+(C_{1}R^{2}-C_{2})\varepsilon ^{\frac{1}{2}%
}+o(\varepsilon ^{\frac{1}{2}}).
\end{equation*}%
Choosing $R>0$ small such that $C_{1}R^{2}-C_{2}<0$ and $\varepsilon >0$
small enough, together with (\ref{e42}), we draw the conclusion. If $N\geq 4
$, then by (\ref{e50}) one has
\begin{equation*}
\widetilde{I}_{r,s}(W_{\varepsilon ,t})\leq \widetilde{m}_{r,s}(a)+\frac{1}{N%
}S^{\frac{N}{2}}s^{1-\frac{N}{2}}-C\varepsilon ^{\frac{2N-q(N-2)}{2}%
}+o(\varepsilon ^{2}|\ln \varepsilon |),
\end{equation*}%
since $\frac{4}{N}<\frac{2N-q(N-2)}{2}<2$ for $2<q<q^{\ast }$. Choosing $%
\varepsilon >0$ small enough, together with (\ref{e42}), we also draw the
conclusion. The proof is complete.
\end{proof}

\begin{theorem}
\label{L3.5} Assume that condition $(V_{0})$ holds and $\widetilde{V}(x)$ is
bounded on $\mathbb{R}^{N}$. Then for any $0<a<\widetilde{a}_{V}$ and $r>%
\widehat{r}_{a}$, problem (\ref{e39}) has a moutain pass type solution $(%
\widehat{\lambda }_{r,s},\widehat{u}_{r,s})$ for almost every $s\in \left[
\frac{1}{2},1\right] $. In particular, there hold $\widehat{u}_{r,s}>0$ and $%
I_{r,s}(\widehat{u}_{r,s})=\widetilde{M}_{r,s}(a)$.
\end{theorem}

\begin{proof}
According to Theorem \ref{L2.2} and Lemma \ref{L3.2}, for almost every $s\in %
\left[ \frac{1}{2},1\right] $, there exists a bounded Palais-Smale sequence $%
\{u_{n}\}\subset S_{r,a}$ satisfying $\widetilde{I}_{r,s}(u_{n})\rightarrow
\widetilde{M}_{r,s}(a)$ and $\widetilde{I}_{r,s}^{\prime
}(u_{n})|_{T_{u_{n}}S_{r,a}}\rightarrow 0$, where $T_{u_{n}}S_{r,a}$ denote
the tangent space of $S_{r,a}$ at $u_{n}$. Up to a subsequence, we assume
that
\begin{equation*}
u_{n}\rightharpoonup \widehat{u}_{r,s}\ \text{in}\ H^{1}(\Omega _{r})\ \text{%
and}\ u_{n}\rightarrow \widehat{u}_{r,s}\ \text{in}\ L^{t}(\Omega _{r})\
\text{for all}\ 2\leq t<2^{\ast }.
\end{equation*}%
It is easy to show that $\widehat{u}_{r,s}\in S_{r,a}$ is a critical point
of $\widetilde{I}_{r,s}$ constrained on $S_{r,a}$. For convince, we set $%
\widehat{w}_{n}=u_{n}-\widehat{u}_{r,s}$. Since $\widetilde{I}_{r,s}^{\prime
}(u_{n})|_{T_{u_{n}}S_{r,a}}\rightarrow 0$, there exists $\lambda _{n}$ such
that $\widetilde{I}_{r,s}^{\prime }(u_{n})-\lambda _{n}u_{n}\rightarrow 0$.
Let $\widehat{\lambda }_{r,s}$ be the lagrange multiplier correspond to $%
\widehat{u}_{r,s}$. Note that
\begin{equation*}
\lim_{n\rightarrow \infty }\int_{\Omega _{r}}V(x)u_{n}^{2}dx=\int_{\Omega
_{r}}V(x)\widehat{u}_{r,s}^{2}dx.
\end{equation*}%
Then it follows from the Br\'{e}zis-Lieb Lemma that
\begin{equation*}
\int_{\Omega _{r}}|\nabla \widehat{w}_{n}|^{2}dx=s\int_{\Omega _{r}}|%
\widehat{w}_{n}|^{2^{\ast }}dx+o_{n}(1)
\end{equation*}%
and
\begin{equation*}
\widetilde{I}_{r,s}(u_{n})=\widetilde{I}_{r,s}(\widehat{u}_{r,s})+\frac{1}{2}%
\int_{\Omega _{r}}|\nabla \widehat{w}_{n}|^{2}dx-\frac{s}{2^{\ast }}%
\int_{\Omega _{r}}|\widehat{w}_{n}|^{2^{\ast }}dx+o_{n}(1).
\end{equation*}%
We assume that $\int_{\Omega _{r}}|\nabla \widehat{w}_{n}|^{2}dx\rightarrow
s\ell \geq 0$ and $\int_{\Omega _{r}}|\widehat{w}_{n}|^{2^{\ast
}}dx\rightarrow \ell \geq 0$. From the definition of $S$ we deduce that
\begin{equation*}
\lim_{n\rightarrow \infty }\int_{\Omega _{r}}|\nabla \widehat{w}%
_{n}|^{2}dx\geq S\left( \lim_{n\rightarrow \infty }\int_{\Omega _{r}}|%
\widehat{w}_{n}|^{2^{\ast }}dx\right) ^{\frac{2}{2^{\ast }}},
\end{equation*}%
which shows that $s\ell \geq S\ell ^{\frac{2}{2^{\ast }}}$. We claim that $%
\ell =0$. Suppose on the contrary that $\ell >0$, then $\ell \geq S^{\frac{N%
}{2}}s^{-\frac{N}{2}}$. This implies that
\begin{equation*}
\widetilde{I}_{r,s}(\widehat{w}_{n})\geq \frac{1}{N}S^{\frac{N}{2}}s^{\frac{%
2-N}{2}}+o_{n}(1)
\end{equation*}%
and hence
\begin{equation*}
\widetilde{I}_{r,s}(u_{n})\geq \widetilde{m}_{r,s}(a)+\frac{1}{N}S^{\frac{N}{%
2}}s^{\frac{2-N}{2}}+o_{n}(1).
\end{equation*}%
From Lemma \ref{L3.4}, we deduce that $\widetilde{I}_{r,s}(u_{n})+o_{n}(1)=%
\widetilde{M}_{r,s}(a)<\widetilde{m}_{r,s}(a)+\frac{1}{N}S^{\frac{N}{2}}s^{%
\frac{2-N}{2}}$. This is a contradiction. Therefore, $\ell =0$ and then $%
u_{n}\rightarrow \widehat{u}_{r,s}$ in $H_{0}^{1}(\Omega _{r})$. The proof
is complete.
\end{proof}

\begin{lemma}
\label{L3.6}
Assume that $\Omega$ is a star-shaped bounded domain.
Let $a>0$  be fixed, the set of solutions $\widehat{u}_{r,s}\in
S_{r,a}$ of problem (\ref{e39}) is bounded uniformly on $s$ and $r$.
\end{lemma}

\begin{proof}
Since $\widehat{u}_{r,s}\in S_{r,a}$ is the solution of problem (\ref{e39})
satisfying $\widetilde{M}_{r,s}(a)=\widetilde{I}_{r,s}(\widehat{u}_{r,s})$,
the following Pohozaev identity holds:%
\begin{eqnarray}
&&\frac{1}{N}\int_{\Omega _{r}}|\nabla \widehat{u}_{r,s}|^{2}dx-\frac{1}{2N}%
\int_{\partial \Omega _{r}}|\nabla \widehat{u}_{r,s}|^{2}(x\cdot n)d\delta -%
\frac{1}{2N}\int_{\Omega _{r}}\widetilde{V}(x)|\widehat{u}_{r,s}|^{2}dx
\notag \\
&=&\frac{s}{2N}\int_{\Omega _{r}}|\widehat{u}_{r,s}|^{2^{\ast }}dx+\frac{%
s(q-2)}{2q}\int_{\Omega _{r}}|\widehat{u}_{r,s}|^{q}dx,  \label{e51}
\end{eqnarray}%
where $n$ denotes the outward unit normal vector on $\partial \Omega _{r}$.
Then by (\ref{e40}) and (\ref{e51}), we get%
\begin{eqnarray*}
&&\frac{1}{N}\int_{\Omega _{r}}|\nabla \widehat{u}_{r,s}|^{2}dx-\frac{1}{2N}%
\int_{\partial \Omega _{r}}|\nabla \widehat{u}_{r,s}|^{2}(x\cdot n)d\delta -%
\frac{1}{2N}\int_{\Omega _{r}}\widetilde{V}(x)|\widehat{u}_{r,s}|^{2}dx \\
&=&\frac{2}{N-2}\left( \frac{1}{2}\int_{\Omega _{r}}|\nabla \widehat{u}%
_{r,s}|^{2}dx+\frac{1}{2}\int_{\Omega _{r}}V(x)|\widehat{u}_{r,s}|^{2}dx-%
\widetilde{M}_{r,s}(a)\right) \\
&&+\frac{s(q-2^{\ast })}{2q}\int_{\Omega _{r}}|\widehat{u}_{r,s}|^{q}dx.
\end{eqnarray*}%
Using this, together with $\Omega _{r}$ being a star-shaped domain, leads to%
\begin{eqnarray*}
\frac{2}{N-2}\widetilde{M}_{r,s}(a) &\geq &\frac{2}{N-2}\left( \frac{1}{2}%
\int_{\Omega _{r}}|\nabla \widehat{u}_{r,s}|^{2}dx+\frac{1}{2}\int_{\Omega
_{r}}V(x)\widehat{u}_{r,s}^{2}dx\right) \\
&&-\frac{1}{N}\int_{\Omega _{r}}|\nabla \widehat{u}_{r,s}|^{2}dx+\frac{1}{2N}%
\int_{\Omega _{r}}(\nabla V\cdot x)\widehat{u}_{r,s}^{2}dx \\
&&-\frac{(2^{\ast }-q)}{2qS^{\frac{2^{\ast }}{2}}}a^{\frac{2q-N(q-2)}{4}%
}\left( \int_{\Omega _{r}}|\nabla \widehat{u}_{r,s}|^{2}dx\right) ^{\frac{%
N(q-2)}{4}} \\
&\geq &\frac{2}{N(N-2)}\int_{\Omega _{r}}|\nabla \widehat{u}%
_{r,s}|^{2}dx-a\left( \frac{1}{2N}\Vert \nabla V\cdot x\Vert _{\infty }+%
\frac{1}{N-2}\Vert V\Vert _{\infty }\right) \\
&&-\frac{(2^{\ast }-q)}{2qS^{\frac{2^{\ast }}{2}}}a^{\frac{2q-N(q-2)}{4}%
}\left( \int_{\Omega _{r}}|\nabla \widehat{u}_{r,s}|^{2}dx\right) ^{\frac{%
N(q-2)}{4}},
\end{eqnarray*}%
which implies that $\widehat{u}_{r,s}\in S_{r,a}$ is bounded uniformly in $s$
and $r$ by Lemma \ref{L3.2} $(ii)$. The proof is complete.
\end{proof}

\begin{theorem}
\label{L3.7} Assume that condition $(V_{0})$ holds and $\widetilde{V}(x)$ is
bounded on $\mathbb{R}^{N}$. Then for any $0<a<\widetilde{a}_{V}$ and $r>\widehat{r}_{a}$, problem (\ref{e4}) admits a solution $\widehat{u}_{r}$ satisfying $\widehat{u}_{r}>0$ in $\Omega _{r}$ and $I_{r}(\widehat{u}%
_{r})>0$ for some Lagrange multiplier $\widehat{\lambda }_{r}\in \mathbb{R}$.
\end{theorem}

\begin{proof}
By Theorem \ref{L3.5}, there exists a nonnegative solution $(\widehat{%
\lambda }_{r,s},\widehat{u}_{r,s})$ to problem (\ref{e39}) for almost every $%
s\in \left[ \frac{1}{2},1\right] $. Moreover, we obtain that $\{\widehat{u}%
_{r,s}\}$ is bounded uniformly on $r$ and $s$ by Lemma \ref{L3.6}. So
similar to the argument in Theorem \ref{L3.5}, there exist $\widehat{u}%
_{r}\in S_{r,a}$ and $\widehat{\lambda }_{r}\in \mathbb{R}$ such that up to
a subsequence,
\begin{equation*}
\widehat{\lambda }_{r,s}\rightarrow \widehat{\lambda }_{r}\ \text{and}\
\widehat{u}_{r,s}\rightarrow \widehat{u}_{r}\ \text{in}\ H_{0}^{1}(\Omega
_{r})\ \text{as}\ s\rightarrow 1.
\end{equation*}%
The proof is complete.
\end{proof}

\textbf{We are ready to prove Theorem \ref{T1.5}:} It is a direct
consequence of Theorem \ref{L3.7}.

\section*{Acknowledgments}

H. Zhang was supported by Hunan Provincial Innovation Foundation For
Postgraduate (No. CX20230219) and the Fundamental Research Funds for the
Central Universities of Central South University (No. 1053320220488). J. Sun
was supported by the National Natural Science Foundation of China (Grant No.
12371174) and Shandong Provincial Natural Science Foundation (Grant No.
ZR2020JQ01).

\end{document}